\documentclass[oneside,english]{amsart}
\usepackage[T1]{fontenc}
\usepackage{verbatim}
\usepackage{units}
\usepackage{amsthm}
\usepackage{amssymb}
\usepackage{graphicx}

\makeatletter
\numberwithin{equation}{section}
\numberwithin{figure}{section}
\theoremstyle{plain}
\newtheorem{thm}{\protect\theoremname}[section]
  \theoremstyle{definition}
  \newtheorem{defn}[thm]{\protect\definitionname}
  \theoremstyle{plain}
  \newtheorem{lem}[thm]{\protect\lemmaname}
  \theoremstyle{remark}
  \newtheorem{rem}[thm]{\protect\remarkname}
  \theoremstyle{remark}
  \newtheorem*{acknowledgement*}{\protect\acknowledgementname}
  \theoremstyle{plain}
  \newtheorem{prop}[thm]{\protect\propositionname}
  \theoremstyle{definition}
  \newtheorem{example}[thm]{\protect\examplename}
  \theoremstyle{plain}
  \newtheorem{cor}[thm]{\protect\corollaryname}

\usepackage{ifpdf} 
\ifpdf 

 \IfFileExists{lmodern.sty}{\usepackage{lmodern}}{}

\title [Divergent Trajectories in Arithmetic Homogeneous Spaces]{Divergent Trajectories in Arithmetic Homogeneous Spaces of Rational Rank Two}

\fi 

\makeatother

\usepackage{babel}
  \providecommand{\acknowledgementname}{Acknowledgement}
  \providecommand{\corollaryname}{Corollary}
  \providecommand{\definitionname}{Definition}
  \providecommand{\examplename}{Example}
  \providecommand{\lemmaname}{Lemma}
  \providecommand{\propositionname}{Proposition}
  \providecommand{\remarkname}{Remark}
\providecommand{\theoremname}{Theorem}

\begin{document}
\global\long\def\bbc{\mathbb{C}}

\global\long\def\bbr{\mathbb{R}}

\global\long\def\bbq{\mathbb{Q}}

\global\long\def\bbz{\mathbb{Z}}

\global\long\def\bbn{\mathbb{N}}

\global\long\def\norm#1{\left\Vert #1\right\Vert }

\global\long\def\frakg{\mathfrak{g}}

\global\long\def\frakh{\mathfrak{h}}

\global\long\def\fraka{\mathfrak{a}}

\global\long\def\frakp{\mathfrak{p}}

\global\long\def\fraku{\mathfrak{u}}

\global\long\def\frakd{\mathfrak{d}}

\global\long\def\frakt{\mathfrak{t}}

\global\long\def\fraks{\mathfrak{s}}

\global\long\def\calb{\mathcal{B}}

\author{Nattalie Tamam}

\date{\today}

\address{Department of Mathematics, Tel Aviv University, Tel Aviv, Israel}

\email{natalita@post.tau.ac.il}
\begin{abstract}
Let $G$ be a semisimple real algebraic group defined over $\bbq$, $\Gamma$
be an arithmetic subgroup of $G$ and $T$ be a maximal $\bbr$-split
torus. A trajectory in $G/\Gamma$ is divergent if eventually it leaves
every compact subset. In some cases there is a finite collection of
explicit algebraic data which account for the divergence. If this
is the case, the divergent trajectory is called obvious. Given a closed
cone in $T$, we study the existence of non-obvious divergent trajectories
under its action in $G/\Gamma$. We get a sufficient condition for
the existence of a non-obvious divergence trajectory in the general
case, and a full classification under the assumption that $\mbox{rank}_{\bbq}G=\mbox{rank}_{\bbr}G=2$. 
\end{abstract}

\maketitle

\section{Introduction }

Let $G$ be a semisimple real algebraic group defined over $\bbq$, $\Gamma$
be an arithmetic subgroup of $G$ and $A\subset G$ be a semigroup.
The action of $A$ on $G/\Gamma$ induces a flow on $G/\Gamma$. The
ergodic theory of these flows is extensively studied and so is the
behavior of a generic trajectory. Some sets of exceptional trajectories
are related to classical problems in number theory (see \cite{key-15}).
A special class of such exceptional trajectories are the divergent
trajectories. It was proved by Dani \cite{key-11} that these trajectories
are related to singular systems of linear forms which are studied
in the theory of Diophantine approximation.

A trajectory $Ax$ in $G/\Gamma$ is called divergent if the map $a\mapsto ax$,
$a\in A$, is proper. In some cases one can find a simple algebraic
reason for the divergence. For example, consider the space of unimodular
lattices $\mbox{SL}_{d}\left(\bbr\right)/\mbox{SL}_{d}\left(\bbz\right)$
with the action of a one-parameter diagonalizable subgroup $\left\{ a_{t}\::\: t\in\bbr\right\} $.
It follows from Mahler’s compactness criterion that if one can find
a non-zero vector $v\in\bbz^{d}$ such that 
\[
\norm{a_{t}x\cdot v}\underset{t\rightarrow\infty}{\longrightarrow}0,
\]
then the trajectory $\left\{ a_{t}x\::\: t\geq0\right\} $ is divergent.
There are more complicated cases in which one has a sequence of different
non-zero vectors $v_{t}$ in $\bbz^{d}$ such that $\norm{a_{t}v_{t}}\rightarrow0$
as $t\rightarrow\infty$. Such trajectories are also divergent but
the divergence is not due to one vector $v$.

Given $A$, some natural questions are: \textbf{Are there divergent
trajectories for the action of $A$ on $G/\Gamma$? Can one always
find a simple reason for the divergence? }

For the case discussed in the above example a classification was proved
by Dani \cite{key-11}. He showed that if $d=2$ all divergent trajectories
are 'degenerate' and if $d\geq3$ then there exists a 'non-degenerate'
divergent trajectory. What is the difference between these two cases?
It seems that the heart of the matter is the rational rank of $G$.
In \cite{key-5} it was conjectured by Weiss that if $A$ is a diagonal
subgroup, then the existence of divergent trajectories and of 'non-obvious'
divergent trajectories depend only on the relation between the dimension
of $A$ and the rational rank of $G$. See Definition \ref{def:obvious} and Definition \ref{degenerate} for the definitions of obvious and degenerate divergent trajectories, respectively, and \cite{key-3,key-5,key-16}
for results regarding the conjecture. 

Assume $\mbox{rank}_{\bbq}G=\mbox{rank}_{\bbr}G=2$. Let $A$ be a two-dimensional closed cone in $G$ (see Definition \ref{def:closed cone} below). Then there are obvious divergent trajectories under the action of $A$. We will show that the existence of non-obvious divergent trajectories under the action of $A$ depends on its intersection with some chambers which are defined using the $\bbq$-fundamental weights of $G$. See \S \ref{sec:The-Main-Results} for the definition of the $\bbq$-fundamental weights. These results generalize Corollary 4.7 and Theorem 4.8 in \cite{key-5}.

We now introduce the terminology and notation we need for stating our main results: Theorem \ref{thm: all obvoius}, which provides a criterion for non-existence
of non-obvious divergent trajectories and Theorem \ref{thm: exists non-obvious}, which provides a criterion for existence of non-obvious divergent trajectories.

\subsection{\label{sec:The-Main-Results}The Main Results }

Let $G$ be a semisimple real algebraic group defined over $\bbq$, $\Gamma$ be an arithmetic subgroup of $G$ (with respect to the given $\bbq$-structure). Denote by $\pi$ the quotient map $G\rightarrow G/\Gamma$, $g\mapsto g\Gamma$.

\begin{defn}
Let $x\in G/\Gamma$ and $A\subset G$. A trajectory $Ax$ is \textbf{divergent} if for any compact subset $K\subset G/\Gamma$ there is a compact $C\subset A$ such that 
\[
h\in A\setminus C\;\Rightarrow\; hx\notin K.
\]
\end{defn}
The easily described divergent trajectories are defined as follows. 
\begin{defn}\label{def:obvious}
Let $g\in G$ and let $A\subset G$ be a semigroup. A trajectory $A\pi\left(g\right)\subset G/\Gamma$
is called an\textbf{ obvious divergent trajectory} if for any unbounded
sequence $\left\{ a_{k}\right\} \subset A$ there is a sub-sequence
$\left\{ a_{k}^{\prime}\right\} \subset\left\{ a_{k}\right\} $, a
$\mathbb{Q}$-representation $\varrho:G\rightarrow\mbox{GL}\left(V\right)$,
and a non-zero $v\in V\left(\mathbb{Q}\right)$ such that 
\[
\varrho\left(a_{k}^{\prime}g\right)v\underset{k\rightarrow+\infty}{\longrightarrow}0.
\]
\end{defn}

A proof that an obvious divergent trajectory is indeed divergent can be found in \cite{key-5}. 

The notion of degenerate divergent trajectories was defined for one parameter semigroups of $G$ in \cite{key-11}. It is similar to the definition of obvious divergent trajectories, but under the consideration of a more restricted class of representations. So, the set of degenerate divergent trajectories is a-priori smaller then the set of obvious divergent trajectories. The following is a generalization of this definition for any semigroup of $G$. 

\begin{defn}\label{degenerate}
Let $g\in G$ and let $A\subset G$ be a semigroup. A trajectory $A\pi\left(g\right)\subset G/\Gamma$ is called a \textbf{degenerate divergent trajectory} if for any unbounded sequence $\left\{a_{k}\right\}\subset$ A there is a sub-sequence $\left\{ a_{k}^{\prime}\right\} \subset\left\{ a_{k}\right\}$, a $\mathbb{Q}$-representation $\varrho:G\rightarrow\mbox{GL}\left(V\right)$, and a nonzero $v\in V\left(\mathbb{Q}\right)$ such that
\begin{enumerate}
\item $\varrho\left(a_{k}^{\prime}g\right)v\underset{k\rightarrow+\infty}{\longrightarrow}0$. 
\item $G_{\left[v\right]}=\left\{ g\in G\::\:\varrho\left(g\right)v\mbox{ is a scalar multiple of }v\right\}$  is a parabolic subgroup of $G$.
\end{enumerate}
\end{defn}
Let $\frakg$ be the Lie algebra of $G$. Equip $\frakg$ with a $\bbq$-structure
which is $\mbox{Ad}\left(\Gamma\right)$-invariant. Fix some rational
basis for $\frakg$ and denote its $\bbz$-span by $\frakg_{\bbz}$.
Let $S$ be a maximal $\bbq$-split torus, and let $T$ be a maximal
$\bbr$-split torus which contains $S$. Denote by $\mathfrak{t}$
and $\mathfrak{s}$ the Lie algebras of $T$ and $S$ respectively. 
\begin{defn}
\label{def:closed cone}A semigroup $A\subset T$ is called a
\textbf{closed cone} if there is a connected subgroup $T_{0}\subset T$, finitely many characters $\lambda_{1},\dots,\lambda_{k}\in T^{*}$
and non-negative $p_{1},\dots,p_{k}$ such that 
\[
A=\left\{a\in T_{0}\::\:\forall i=1,\dots,k,\;\lambda_{i}\left(a\right)\geq p_{i}\right\} .
\]
\end{defn}
It is shown in \cite{key-5} that obvious divergent trajectories for
the action of closed cones are determined by finitely many rational
representations and finitely many rational vectors.

Denote by $\Phi_{\bbq}$ the set of $\bbq$-roots in $\mathfrak{s}^{*}$. 
Denote by $W\left(\Phi_{\bbq}\right)$ the Weyl group associated with
$\Phi_{\bbq}$, i.e. the group generated by the reflections $\omega_{\lambda}$,
$\lambda\in\Phi_{\bbq}$, defined by 
\begin{equation}
\omega_{\lambda}\left(\chi\right)=\chi-\left\langle \chi,\lambda\right\rangle\lambda\label{eq: Weyl group def}
\end{equation}
for any characters $\chi$ in $\fraks^{*}$ ($\left\langle \cdot,\cdot \right\rangle$ will be defined in \S \ref{sub:real rep}).

There exists a subset $\Delta_{\bbq}\subset\Phi_{\bbq}$
such that any $\lambda\in\Phi_{\bbq}$ can be expressed uniquely as
a linear combination 
\begin{equation}
\lambda=\sum_{\alpha\in\Delta_{\bbq}}m_{\alpha}\left(\lambda\right)\alpha\label{eq: lambda=00003Dsum n(lambda) mu}
\end{equation}
 in which each $m_{\alpha}\left(\lambda\right)\in\bbz$ and either
all $m_{\alpha}\left(\lambda\right)\geq0$ or all $m_{\alpha}\left(\lambda\right)\leq0$
. The set $\Delta_{\bbq}$ is called a $\bbq$-simple system and each
$\alpha\in\Delta_{\bbq}$ is called a $\bbq$-simple root. For $\lambda_{1},\lambda_{2}\in\fraks^{*}$
we write $\lambda_{1}\geq\lambda_{2}$ if $\lambda_{1}-\lambda_{2}$
can be written as a linear combination of $\bbq$-simple roots with
non-negative coefficients, and $\lambda_{1}>\lambda_{2}$ if $\lambda_{1}\geq\lambda_{2}$
and $\lambda_{1}\neq\lambda_{2}$. 

Fix
\begin{equation}
\Delta_{\bbq}=\left\{ \alpha_{1},\dots,\alpha_{r}\right\} .\label{eq: Delta def}
\end{equation}

Let $\chi_{1},\dots,\chi_{r}\in\fraks^{*}$ be the\textbf{ $\bbq$-fundamental
weights} of $\frakg$, i.e. for any $1\leq i,j\leq r$ 
\begin{equation}
\left\langle \chi_{i},\alpha_{j}\right\rangle =\delta_{i,j}\label{eq:fundamental weights}
\end{equation}
(Kronecker delta).

For $\lambda\in\frakt^{*}$ denote by $\lambda\mid_{\mathfrak{s}}$
the restriction of $\lambda$ to $\mathfrak{s}$.

\begin{rem}
	\label{rem:chi_tilde}
	According to \cite[\S8.5]{key-1} that exists $\frakt_{0}$ such that $\frakt$ is a direct sum of $\frakt_{0}$ and $\fraks$. Thus, for any character $\chi\in\fraks^*$ one can define $\tilde{\chi}\in\frakt^*$ uniquely by  $\tilde{\chi}\mid_{\fraks}=\chi$ and $\tilde{\chi}\mid_{\frakt_{0}}=0$. 
\end{rem}

Let $\norm{\cdot}$ be a norm on $\frakt$. 

Our main result provides a necessary condition for the existence of non-obvious divergence trajectories under the action of closed cones in $T$ on $G/\Gamma$. 
\begin{thm}
	\label{thm: all obvoius}Assume $\mbox{rank}_{\bbq}G=\mbox{rank}_{\bbr}G=2$.
	Let
	\begin{eqnarray*}
		\mathfrak{a}^{+} & = & \left\{ t\in\mathfrak{t}\::\:\chi_{1}\left(t\right)\geq0,\;\omega_{\alpha_{1}}\left(\chi_{1}\right)\left(t\right)\geq0\right\} ,\\
		A^{+} & = & \exp\left(\mathfrak{a}^{+}\right).
	\end{eqnarray*}
	Then for any closed cone $A\supset A^{+}$, there are only degenerate
	divergent trajectories for the action of $A$ on $G/\Gamma$. 
\end{thm}

Our second result provides a sufficient condition for the existence of a non-obvious divergence trajectory under the action of closed cones in $T$ on $G/\Gamma$. In particular, it shows that when $\mbox{rank}_{\bbq}G=\mbox{rank}_{\bbr}G=2$, any closed cone which does not satisfy the assumption of Theorem \ref{thm: all obvoius},
does not satisfy its conclusion as well. See Figure \ref{fig: A_epsilon}. 

\begin{thm}
\label{thm: exists non-obvious}Assume $\mbox{rank}_{\bbq}G\geq2$.
Let $\epsilon>0$ and let 
\begin{eqnarray*}
\mathfrak{a}_{\epsilon} & = & \left\{ t\in\mathfrak{t}\::\:\tilde{\chi_1}\left(t\right)\geq\epsilon\norm t,\;\tilde{\chi_2}\left(t\right)\geq\epsilon\norm t\right\} ,\\
A_{\epsilon} & = & \exp\left(\mathfrak{a}_{\epsilon}\right).
\end{eqnarray*}
Then for any unbounded closed cone $A\subset A_{\epsilon}$ there exists a non-obvious
divergent trajectory for the action of $A$ on $G/\Gamma$. 
\end{thm}

\begin{figure}[h]
\includegraphics[scale=0.75]{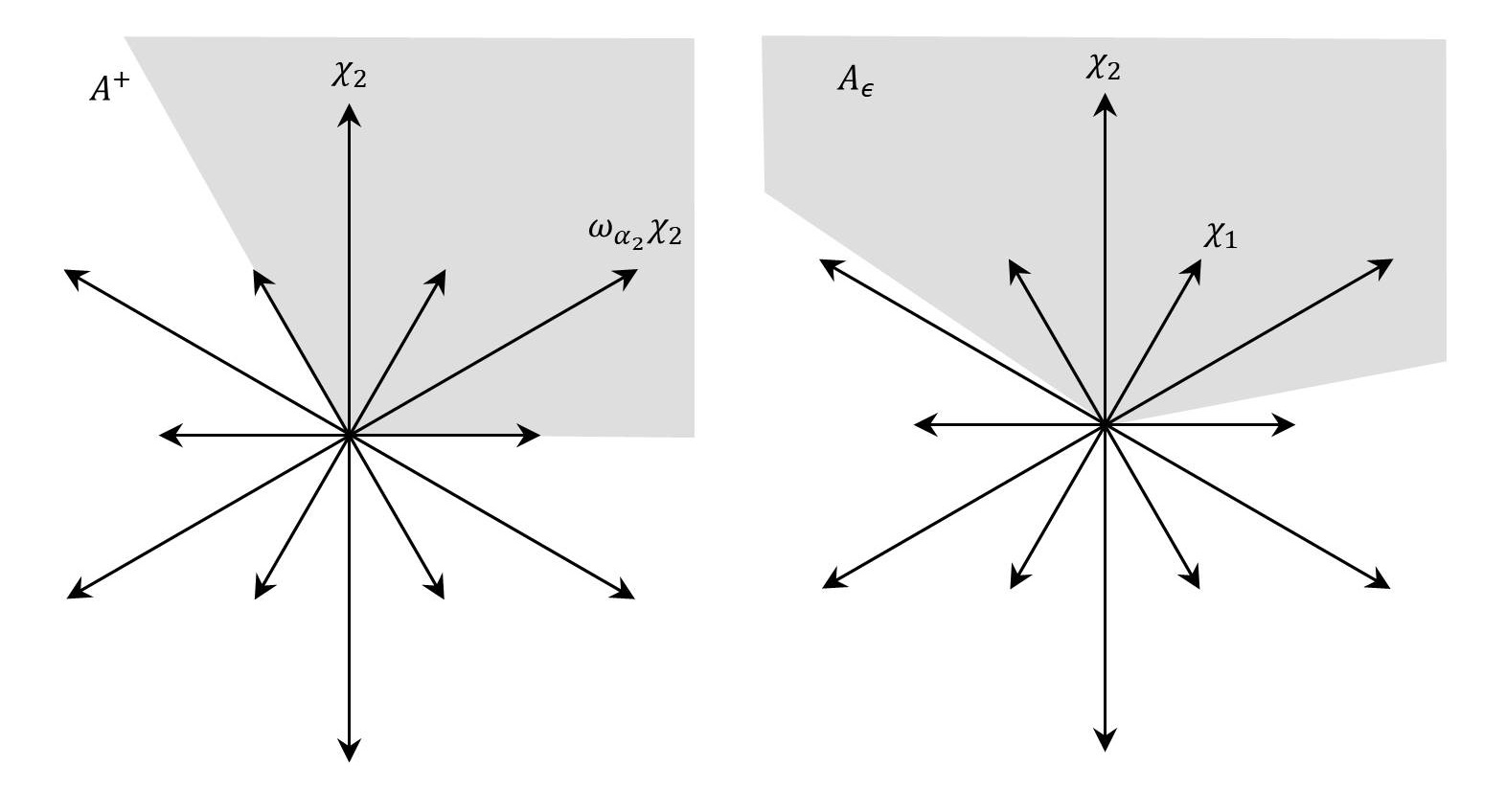}

\protect\caption{\label{fig: A_epsilon}Depiction of $A^{+}$ (as defined in
Theorem \ref{thm: all obvoius}) and $A_{\epsilon}$ (as defined in
Theorem \ref{thm: exists non-obvious}) for the standard basis of
$\protect\frakg_{2}$. }
\end{figure}

\begin{rem}
A different choice of a simple system or a different indexing of it
in (\ref{eq: Delta def}) will result in a different closed cone in
Theorem \ref{thm: all obvoius}. In particular, when $\mbox{rank}_{\bbq}G=2$
the two indexing options of the standard basis will result in two
different closed cones which are not images of each other by the action
of the Weyl group. The Weyl group acts simply transitively on simple
systems. Thus, up to the action of the Weyl group, these two examples
give all the closed cones which may appear in Theorem \ref{thm: all obvoius}.
See figure \ref{fig: A_cases}.
\end{rem}
\begin{figure}[h]
\includegraphics[scale=0.75]{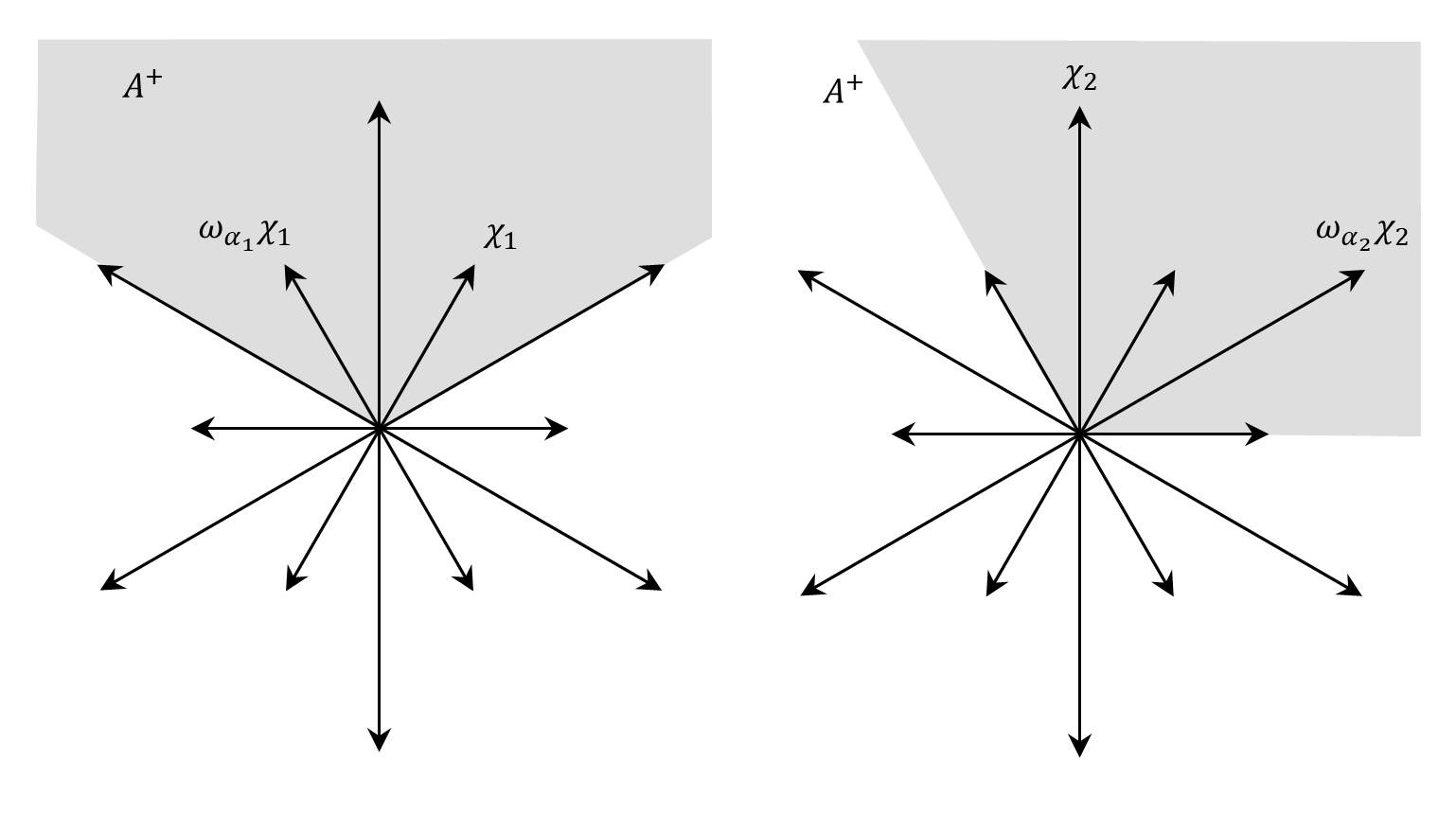}

\protect\caption{\label{fig: A_cases}Depiction of $A^{+}$ (as defined in Theorem
\ref{thm: all obvoius}) for the two indexing options of the standard
basis of $\protect\frakg_{2}$.}
\end{figure}

\begin{rem}
\label{rem:reducible}It follows from \cite[\S 2.15]{key-27} that up to a finite index $G$ is an almost 
direct product of its $\bbq$-almost simple $\bbq$-factors. By \cite[\S 5.11]{key-27}
the relative root system of each $\bbq$-almost simple $\bbq$-factor is irreducible. 
Hence, if $\Phi_{\bbq}$ is a reducible root system of rational rank two, 
then up to a finite index $G$ is an almost direct product of $\bbq$-subgroups of $G$,
each of rational rank one. In that case, Theorem \ref{thm: all obvoius} follows from \cite[Theorem 6.1]{key-11}. 
Hence, in the proof of Theorem \ref{thm: all obvoius}
we may assume $\Phi_{\bbq}$ is irreducible. 
\end{rem}

\subsection{Structure}

Theorem \ref{thm: all obvoius} is proved in \S \ref{sec:Proof-of-Theorem2}.
We start by looking at a divergent trajectory under the action of a closed cone. In \S \ref{sec:Compactness-criterion} a compactness criterion that can be deduced from \cite[\S 3]{key-3} is stated. Using the compactness
criterion we can attach to each element in the cone its 'reason for divergence'. The rank assumption in the theorem implies that there are essentially two such reasons. We denote elements in the corresponding sets by $\frakd_{1},\frakd_{2}$. Then, up to a compact set, $\frakd_{1}$ and $\frakd_{2}$ form a cover of $\mbox{Lie}\left(A\right)$. We then use a topological property of $\frakd_{1}$ in order to prove the existence of an unbounded connected component in $\frakd_{2}$. This topological property is stated in
Theorem \ref{thm: conn component intersects both rays}. As an important step in the proof of Theorem \ref{thm: conn component intersects both rays}, in \S \ref{sec:highest weight} we prove Theorem \ref{thm:weyl elem norm bound}
regarding the norm of the image of a highest weight vector. In the proof of Theorem \ref{thm: conn component intersects both rays} we also use properties of real representations which are proved in \S \ref{sub:weights} as well as corollaries of the compactness criterion. 
Once we know that there exist unbounded connected components in $\frakd_{1},\frakd_{2}$, we use Proposition \ref{prop: unbounded element} to find a 'nice' vector in each unbounded component. By using similar arguments if necessary, we then show that there is a finite set of representations and vectors (which were chosen to satisfy the second condition of Definition \ref{degenerate}) which 'cause' the divergence, proving the trajectory is a degenerate divergent trajectory. 

Theorem \ref{thm: exists non-obvious} is proved in \S \ref{sec:Proof-of-Theorem 1}.
This proof can be read independently of previous sections (apart from notation which appear in \S \ref{sec: lie algebras} and \S \ref{sec:Compactness-criterion}). 

\begin{acknowledgement*}
I would like to thank my advisor Barak Weiss for his guidance, support
and encouragement as well as many helpful discussions. 

This research was supported by the Ministry of Science, Technology
and Space, Israel, by the European Research Council (ERC) starter
grant DLGAPS 279893, and by the Israel Science Foundation (ISF) grant
2095/15. 
\end{acknowledgement*}

\section{\label{sec: lie algebras}Preliminary Results}

\subsection{\label{sub:real rep}Real Representations}

Denote by $\Phi_{\bbr}$ the set of $\bbr$-roots. For $\lambda\in\Phi_{\bbr}$ denote by $\frakg_{\lambda}$ the
$\bbr$-root space for $\lambda$. 

Let $\kappa$ be the Killing form on $\frakg$ and $\theta$ be the
Cartan involution associated with $\kappa$. For $\lambda\in\mathfrak{t}^{*}$ let $t_{\lambda}\in\mathfrak{t}$ be determined by $\lambda\left(t\right)=\kappa\left(t_{\lambda},t\right)$
for all $t\in\mathfrak{t}$. For $\lambda_{1},\lambda_{2}\in\mathfrak{t}^{*}$ let 
\[
\left(\lambda_{1},\lambda_{2}\right) =\kappa\left(t_{\lambda_{1}},t_{\lambda_{2}}\right).
\]

Let $\rho:\frakg\rightarrow\mathfrak{gl}\left(V\right)$ be an  $\bbr$-representation of $\frakg$. For $\lambda\in\frakt^{*}$ denote
\[
V_{\rho,\lambda}=\left\{ v\in V\::\:\forall t\in\frakt\quad\rho\left(t\right)v=\lambda\left(t\right)v\right\} .
\]
If $V_{\rho,\lambda}\neq\left\{ 0\right\} $, then $\lambda$ is
called an \textbf{$\bbr$-weight for $\rho$}. Denote by $\Phi_{\rho}$
the set of $\bbr$-weights for $\rho$. For any $\lambda\in\Phi_{\rho}$,
$V_{\rho,\lambda}$ is called \textbf{the $\bbr$-weight
	vector space for $\lambda$}, and members of $V_{\rho,\lambda}$
are called \textbf{$\bbr$-weight vectors for $\lambda$}.

Roots and weights are related by 
\begin{equation}
\rho\left(\frakg_{\lambda}\right)V_{\rho,\chi}\subset V_{\rho,\lambda+\chi},\quad\lambda\in\Phi_{\bbr},\;\chi\in\Phi_{\rho}.\label{eq: lambda acts on chi}
\end{equation}

For any $\lambda,\mu\in\frakt^{*}$ denote 
\begin{equation}
\left\langle\lambda,\mu\right\rangle=2\frac{\left(\lambda,\mu\right) }{\left(\lambda,\lambda\right)}.\label{eq: l_lambda,mu def}
\end{equation}
By \cite[\S 5]{key-27}, for any $\lambda\in\Phi_{\bbr}$, real representation
$\rho$, and $\mu\in\Phi_{\rho}$ 
\begin{equation}
\left\langle\lambda,\mu\right\rangle\mbox{ is an integer.}\label{eq: l integer}
\end{equation}

The Lie algebra $\mathfrak{sl}\left(2,\bbr\right)$ has a basis $H,X,Y$
which satisfies 
\begin{equation}
\mbox{ad}\left(H\right)X=2X,\quad\mbox{ad}\left(H\right)Y=-2Y,\quad\mbox{ad}\left(X\right)Y=-H.\label{eq:sl2} 
\end{equation}

\begin{lem}
	\label{lem: sl(2)}\cite[\S VIII.1.2]{key-25}Let $\left(V,\rho\right)$ be an irreducible real representation of $\mathfrak{sl}\left(2,\bbr\right)$ with dimension $k$.
	Assume $V$ is generated by an element $v\in V$ such that $\rho\left(Y\right)v=0$ and
	$\rho\left(H\right)v=\lambda v$ for some $\lambda\neq0$. Then $-\lambda=k-1$.
\end{lem}

The following lemma follows from the proof of Proposition 6.52 in  \cite{key-7}. 

\begin{lem}
	\label{lem: 1} Let $\lambda\in\Phi_{\bbr}$ and $E\in\frakg_{\lambda}$ be non-zero. 
	For some $c>0$ the elements $X=cE$, $Y=\theta X$, and $H=\frac{2}{\left( \lambda,\lambda\right)}t_\lambda$ satisfy \eqref{eq:sl2}. In particular, 
	$\mathfrak{v}=\bbr H \oplus \bbr X \oplus \bbr Y$ is a Lie subalgebra of $\frakg$ isomorphic to $\mathfrak{sl}\left(2,\bbr\right)$.
\end{lem}


\begin{lem}
	\label{lem: =00005BX_lambda ,X_mu=00005D}Let $(\rho,V)$ be a finite dimensional real
	representation of $\frakg$, $\lambda\in\Phi_{\bbr}$, $\mu\in\Phi_{\rho}$, $E\in\frakg_{\lambda}$, and $v\in V_{\rho,\mu}$.
	Assume $\lambda-\mu\notin\Phi_{\rho}$. Then $\left\langle\lambda,\mu\right\rangle\leq 0$ and if  
	\begin{equation}
	\rho\left(E\right)^{l}v=0,\label{eq:lambda l mu=0}
	\end{equation}
	for some non-negative $l\leq -\left\langle\lambda,\mu\right\rangle$, then $E=0$ or $v=0$. \end{lem}
\begin{proof}
	According to \cite[Lem. 10.3]{ha}, $\lambda-\mu\notin\Phi_{\rho}$ implies  $\left\langle\lambda,\mu\right\rangle\leq 0$. 
	
	Assume by contradiction that there exists $l\leq \left\langle\lambda,\mu\right\rangle$,
	such that (\ref{eq:lambda l mu=0}) is satisfied.
	Let $X,Y,H,\mathfrak{v}$ be defined as in Lemma \ref{lem: 1}  for $E$. 
	
	Denote 
	\[
	W=\mbox{span}_{\bbr}\left\{ v,\rho\left(X\right)v,\dots,\rho\left(X\right)^{l-1}v\right\} .
	\]
	It follows from \eqref{eq: lambda acts on chi} and $\lambda-\mu\notin\Phi_{\rho}$ that $\rho\left(Y\right)v=0$.
	Thus, $W$ is invariant under the action of $\mathfrak{v}$. Therefore, $(\rho,W)$ is an irreducible real representation of 
	$\mathfrak{v}$ with dimension $l$. In addition, we have $\rho\left(H\right)v=\mu\left(H\right)v$
	and 
	\[
	\mu\left(H\right)=\kappa\left(t_{\mu},H\right)=\kappa\left(t_{\mu},\frac{2}{\left( \lambda,\lambda\right)}t_\lambda\right)=2\frac{\left( \mu,\lambda\right) }{\left( \lambda,\lambda\right)}=\left\langle \lambda,\mu\right\rangle.
	\]
	Hence, by Lemma \ref{lem: sl(2)} we get $-\left\langle \lambda,\mu\right\rangle= l-1$,
	a contradiction. 
\end{proof}

\begin{lem}
	\label{lem: =00005BX_lambda ,X_mu=00005D general}Let $\lambda,\mu\in\Phi_{\bbr}$, $X_{\lambda}\in\frakg_{\lambda}$,
	$X_{2\lambda}\in\frakg_{2\lambda}$, and $X_{\mu}\in\frakg_{\mu}$.
	Assume $\lambda-\mu\notin\Phi_{\bbr}\cup\left\{ 0\right\}$. Then,  $l=-\left\langle\lambda,\mu\right\rangle\geq0$, and 
	\begin{equation}
	\sum_{k\leq\left[\frac{l}{2}\right]}\frac{1}{k!}\cdot\frac{1}{\left(l-2k\right)!}\cdot\mbox{ad}\left(X_{2\lambda}\right)^{k}\mbox{ad}\left(X_{\lambda}\right)^{l-2k}X_{\mu}=0,\label{eq: =00005BX_lambda ,X_mu=00005D general}
	\end{equation}
	implies that $X_{\lambda}=X_{2\lambda}=0$ or $X_{\mu}=0$. \end{lem}
\begin{proof}
	It is again follow from \cite[Lem. 10.3]{ha} and  $\lambda-\mu\notin\Phi_{\rho}$ that $\left\langle\lambda,\mu\right\rangle\leq 0$.	
	
	Note that if $2\lambda\notin\Phi_{\bbr}$, we can deduce the conclusion
	of the lemma from Lemma \ref{lem: =00005BX_lambda ,X_mu=00005D}.
	Hence we may assume $2\lambda\in\Phi_{\bbr}$. 
	According to (\ref{eq: l integer}), $\left\langle2\lambda,\mu\right\rangle$ is an integer. By the linearity of $\left(\cdot,\cdot\right)$ and (\ref{eq: l_lambda,mu def}) we have  $\left\langle\lambda,\mu\right\rangle=2\left\langle2\lambda,\mu\right\rangle$.
	Thus, $l$ is even. If $l=0$, then (\ref{eq: =00005BX_lambda ,X_mu=00005D general})
	implies $X_{\mu}=0$. Thus, we may assume $l>0$. Then, according
	to \cite[\S VI.1.3]{key-12} $l\in\left\{ 2,4\right\} $.
	
	First we will prove that for any $X_{\lambda}\in\frakg_{\lambda}$,
	$X_{2\lambda}\in\frakg_{2\lambda}$, and $X_{\mu}\in\frakg_{\mu}$,
	\begin{equation}
	\mbox{ad}\left(X_{2\lambda}\right)X_{\mu}+\mbox{ad}\left(X_{\lambda}\right)^{2}X_{\mu}=0\mbox{ implies }X_{\lambda}=X_{2\lambda}=0\mbox{ or }X_{\mu}=0.\label{eq: =00005Blambda,mu=00005D gen first}
	\end{equation}
	Assume by contradiction that there exist $0\neq X_{\lambda}\in\frakg_{\lambda}$,
	and $X_{2\lambda}\in\frakg_{2\lambda}$, $X_{\mu}\in\frakg_{\mu}$
	not both zero, such that $\mbox{ad}\left(X_{2\lambda}\right)X_{\mu}+\mbox{ad}\left(X_{\lambda}\right)^{2}X_{\mu}=0$.
	Then 
	\begin{equation}
	\mbox{ad}\left(\theta X_{\mu}\right)\left(\mbox{ad}\left(X_{2\lambda}\right)X_{\mu}+\mbox{ad}\left(X_{\lambda}\right)^{2}X_{\mu}\right)=0.\label{eq: k=00003D2 case 1}
	\end{equation}
	By the assumption $\lambda-\mu\notin\Phi_{\bbr}$, which implies $2\lambda-\mu\notin\Phi_{\bbr}$.
	Hence $\mbox{ad}\left(\theta X_{\mu}\right)$ commutes with $\mbox{ad}\left(X_{\lambda}\right)$
	and $\mbox{ad}\left(X_{2\lambda}\right)$. Therefore (\ref{eq: k=00003D2 case 1})
	implies 
	\[
	\mbox{ad}\left(X_{2\lambda}\right)\mbox{ad}\left(\theta X_{\mu}\right)X_{\mu}+\mbox{ad}\left(X_{\lambda}\right)^{2}\mbox{ad}\left(\theta X_{\mu}\right)X_{\mu}=0.
	\]
	By Lemma \ref{lem: 1}, the definition of $t_{\mu}$, and the anti-symmetry
	of the Lie brackets we arrive at 
	\begin{eqnarray*}
		0 &=& \mbox{ad}\left(X_{2\lambda}\right)t_{\mu}+\mbox{ad}\left(X_{\lambda}\right)^{2}t_{\mu}\\
		&=& \lambda\left(t_{\mu}\right)\left(2X_{2\lambda}+\mbox{ad}\left(X_{\lambda}\right)X_{\lambda}\right)\\
		&=&2\left(\lambda,\mu\right) X_{2\lambda}.
	\end{eqnarray*}
	Thus $X_{2\lambda}=0$. Now, according to Lemma \ref{lem: =00005BX_lambda ,X_mu=00005D},
	either $X_{\lambda}=0$ or $X_{\mu}=0$, a contradiction. 
	
	Assume $X_{\lambda}\in\frakg_{\lambda}$, $X_{2\lambda}\in\frakg_{2\lambda}$
	and $X_{\mu}\in\frakg_{\mu}$ satisfy (\ref{eq: =00005BX_lambda ,X_mu=00005D general}).
	If $l=2$, then (\ref{eq: =00005BX_lambda ,X_mu=00005D general})
	implies 
	\begin{equation}
	\mbox{ad}\left(X_{2\lambda}\right)X_{\mu}+\frac{1}{2}\mbox{ad}\left(X_{\lambda}\right)^{2}X_{\mu}=0.\label{eq: k=00003D2 case}
	\end{equation}
	By replacing $X_{\lambda}$ with $\frac{1}{\sqrt{2}}X_{\lambda}$,
	the conclusion of the lemma follows from (\ref{eq: =00005Blambda,mu=00005D gen first}). 
	
	If $l=4$, then (\ref{eq: =00005BX_lambda ,X_mu=00005D general})
	implies 
	\begin{equation}
	\frac{1}{2}\cdot\mbox{ad}\left(X_{2\lambda}\right)^{2}X_{\mu}+\frac{1}{2}\cdot\mbox{ad}\left(X_{2\lambda}\right)\mbox{ad}\left(X_{\lambda}\right)^{2}X_{\mu}+\frac{1}{4!}\cdot\mbox{ad}\left(X_{\lambda}\right)^{4}X_{\mu}=0.\label{eq: k=00003D4 case}
	\end{equation}
	One can find $c_{1},c_{2}>0$ such that the LHS of (\ref{eq: k=00003D4 case})
	is equal to 
	\[
	\frac{1}{2}\left(\mbox{ad}\left(X_{2\lambda}\right)+c_{1}\mbox{ad}\left(X_{\lambda}\right)^{2}\right)\left(\mbox{ad}\left(X_{2\lambda}\right)X_{\mu}+c_{2}\mbox{ad}\left(X_{\lambda}\right)^{2}X_{\mu}\right).
	\]
	Thus, by replacing $X_{\lambda}$ with $\sqrt{c_{1}}X_{\lambda}$
	or $\sqrt{c_{2}}X_{\lambda}$, the conclusion of the lemma follows
	from (\ref{eq: =00005Blambda,mu=00005D gen first}). 
\end{proof}

\subsection{\label{sub:weights}The Fundamental Weights}

Recall that $\Delta_{\bbq}$ is a $\bbq$-simple system and  $\chi_{1},\dots,\chi_{r}$ are the corresponding $\bbq$-fundamental weights 
(defined in \S \ref{sec:The-Main-Results}). 

As in the previous section, $\Phi_{\bbr}$ is the set of $\bbr$-roots. 
According to \cite[\S 21.8]{key-1} there exists an $\bbr$-simple
system $\Delta_{\bbr}\subset\Phi_{\bbr}$ such that the order on $\Phi_{\bbr}$ defined using this simple system satisfies 
\begin{equation}
\alpha>\beta\quad\Rightarrow\quad\alpha\mid_{\fraks}\geq\beta\mid_{\fraks}.\label{eq:order relation between roots}
\end{equation}

Denote by $\Phi_{\bbr}^{+}$ the set of positive $\bbr$-roots, i.e. the roots $\lambda\in\Phi_{\bbr}$ such that $\lambda>0$. 

If the only multiples of a $\bbq$-root $\lambda$ in $\Phi_{\bbq}$ are $\pm\lambda$, then
$\lambda$ is called \textbf{reduced}. If all $\lambda\in\Phi_{\bbq}$
are reduced, then $\Phi_{\bbq}$ is called \textbf{reduced}. 

For any $\alpha\in\Delta_{\bbq}$
let 
\begin{equation}
\Phi_{\alpha}=\left\{\beta\in\Phi_{\bbr}\::\: \Phi_\bbq \mbox{ is reduced and }\beta\mid_{\fraks}\geq\alpha,\mbox{ or }\Phi_\bbq \mbox{ is non-reduced and }\beta\mid_{\fraks}\geq 2\alpha\right\}
\label{eq:Phi_chi}
\end{equation}
and 
\begin{equation}
\chi_{\alpha}=\sum_{\beta\in\Phi_{\alpha}}\beta.
\label{eq:chi_alpha}
\end{equation}

\begin{lem}
	\label{lem:fundam_rep}For any $\alpha\in\Delta_{\bbq}$,  $\beta\in\Phi_{\bbr}^{+}$ 
	\[
	\left\langle \chi_{\alpha},\beta\right\rangle 
	\begin{cases}
	\geq 0, & \text{if } \beta\mid_{\fraks}\geq\alpha\\
	=0, & \text{else} 
	\end{cases}
	\]
\end{lem}

\begin{proof}
	Let $\beta\in\Phi_{\bbr}^{+}$ such that $\beta\mid_\fraks\ngeq\alpha$. Then, equation \eqref{eq: Weyl group def} implies that for any $\lambda\in\Phi_{\alpha}$ we have $\omega_{\beta}\left(\lambda\right)\in\Phi_{\alpha}$. 
	Hence, $\omega_{\beta}$ leaves $\Phi_{\alpha}$ invariant. Thus,  \eqref{eq:chi_alpha} implies  $\omega_{\beta}\left(\chi_{\alpha}\right)=\chi_{\alpha}$. If follows from \eqref{eq: Weyl group def} that $\left\langle\chi_{\alpha},\beta\right\rangle =0$. Proving the second part of the lemma.
	
	Let $\beta\mid_\fraks\geq\alpha$. Then for any $\lambda\in\Phi_{\alpha}$ either $\omega_{\beta}\left(\lambda\right)\in\Phi_{\alpha}$ or $\omega_{\beta}\left(\lambda\right)<\lambda$. Therefore
	\begin{equation}
	\omega_{\beta}\left(\chi_{\alpha}\right)\leq\chi_{\alpha}.
	\label{eq:omega_beta leq}
	\end{equation}
	The second part of the lemma now follows from \eqref{eq: Weyl group def}. 
\end{proof}

\begin{rem}
	\label{rem:chi_alpha_i}In a similar way to Remark \ref{rem:chi_tilde}, one can deduce from Lemma \ref{lem:fundam_rep} that for any $1\leq i\leq r$, $\chi_{\alpha_i}$ is a scalar multiplication of $\tilde{\chi_i}$. 
\end{rem}

\begin{lem}
	\label{lem: not both pos}Assume $\mbox{rank}_{\bbq}G=2$ and $\Phi_{\bbq}$
	is irreducible. For any $\omega\in W\left(\Phi_{\bbq}\right)$ there
	exist $a,b$, not both positive, such that $\omega\left(\chi_{1}\right)=a\chi_{1}+b\omega_{\alpha_{1}}\left(\chi_{1}\right)$.
\end{lem}
\begin{proof}
	Let $\omega\in W\left(\Phi_{\bbq}\right)$. 
	Since the statement is satisfied for the identity element in $W\left(\Phi_{\bbq}\right)$ and for $\omega_{\alpha_1}$, we may assume $\omega$ is neither of them. 
	
	We say that a weight is dominant if it is a non-negative linear combination of the $\bbq$-fundamental weights. 
	
	It is known (see \cite[\S13.2]{hum}) that any rational weight is conjugated under the Weyl group to exactly one dominant weight. Since $\chi_1$ is a dominant weight, there exist $a_1,b_1$ not both positive such that 
	\begin{equation}
	\omega\left(\chi_{1}\right)=a_1 \chi_{1} + b_1 \chi_{2}
	\label{eq:two decomp}
	\end{equation}
	
	Note that $\left\{\omega_{\alpha_{1}}\left(\chi_{1}\right), \chi_{2}\right\}$ are the fundamental weights with respect to the simple system $\omega_{\alpha_{1}}\left(\Delta\right)$. Thus, in a similar way to \eqref{eq:two decomp} one can get that there exist $a_2,b_2$ not both positive
	such that 	
	\begin{equation}
	\omega\left(\chi_{1}\right)=a_2 \omega_{\alpha_{1}}\left(\chi_{1}\right) + b_2 \chi_{2}.
	\label{eq:two decomp2}
	\end{equation}
	
	Using \eqref{eq:fundamental weights} it can be checked that for some $c>0$, 
	\begin{equation}
	\chi_2=c\left(\omega_{\alpha_1}\left(\chi_1\right) + \chi_1\right)
	\label{eq:chi2}
	\end{equation}
	
	It follows from \eqref{eq:fundamental weights}, \eqref{eq:two decomp}, and \eqref{eq:two decomp2} 
	\begin{eqnarray*}
		a_1 &=&\langle \omega\left(\chi_{1}\right),\alpha_{1} \rangle\\
		&=&-\langle \omega\left(\chi_{1}\right),-\alpha_{1}\rangle\\
		&=&-\langle \omega\left(\chi_{1}\right),\omega_{\alpha_{1}}\left(\alpha_{1}\right)\rangle\\
		&=&-a_2 
	\end{eqnarray*}	Hence, $b_1, b_2$ are not both positive. Without loss of generality, assume $b_1\leq 0$. Then, by \eqref{eq:two decomp} and \eqref{eq:chi2} we arrive at \[
	\omega\left(\chi_{1}\right)=\left(a_1+c b_1\right)\chi_{1} + c b_1 \omega_{\alpha_1}\left(\chi_1\right),
	\]
	with $cb_1\leq0$.	
\end{proof}

\subsection{\label{sub:representations}Strongly Rational Representations}

Let $\varrho:G\rightarrow\mbox{GL}\left(V\right)$ be an $\bbr$-representation and  $\rho:\frakg\rightarrow\mathfrak{gl}\left(V\right)$ be its derivative. It will be convenient notationally to refer to $\Phi_{\rho}$ as $\Phi_{\varrho}$ and for any $\lambda\in\Phi_{\rho}$ to $V_{\rho,\lambda}$ as $V_{\varrho,\lambda}$ (in the notation of \S \ref{sub:real rep}). Moreover, elements of $\Phi_{\varrho}$ are called \textbf{$\bbr$-weights for $\varrho$}, for any $\lambda\in\Phi_{\varrho}$, $V_{\varrho,\lambda}$ is called \textbf{the $\bbr$-weight vector space for $\lambda$}, and members of $V_{\varrho,\lambda}$
are called \textbf{$\bbr$-weight vectors for $\lambda$}.

An $\bbr$-weight $\chi$ for $\varrho$ is called an \textbf{$\bbr$-highest weight for $\varrho$} if any $\lambda\in\Phi_{\varrho}$ satisfies $\lambda\leq\chi$. 

\begin{defn}
	\label{def:strongly rational}A finite-dimensional $\bbr$-representation
	$\varrho:G\rightarrow\mbox{GL}\left(V\right)$ is called \textbf{strongly
	rational over $\bbr$} if there is an $\bbr$-highest weight for $\varrho$ and the $\bbr$-weight vector space for the $\bbr$-highest weight is of dimension one. It is called \textbf{strongly
	rational over $\bbq$} if it is strongly rational over $\bbr$, defined over $\bbq$, and the $\bbr$-weight vector space for the $\bbr$-highest weight is also defined over $\bbq$.
\end{defn}

For $\alpha\in\Delta_{\bbq}$ let $\fraku_{\alpha}=\bigoplus_{\lambda\in\Phi_{\alpha}}\frakg_{\lambda}$,
$d_{\alpha}=\dim\fraku_{\alpha}$, and $\varrho_{\alpha}:G\rightarrow\mbox{GL}\left(\bigwedge^{d_{\alpha}}\frakg\right)$
be the $d_{\alpha}$-th exterior power of the adjoint representation (see \S \ref{sub:weights} for the definition of $\Phi_{\alpha}$ and   \cite{key-8,key-9,key-10}
for the definition and properties of the exterior power). For any $\alpha\in\Delta_\bbq$, $\varrho_{\alpha}$ is strongly rational over $\bbq$ with $\bbr$-highest
weight $\tilde{\chi}_{\alpha}$ and $V_{\varrho_{\alpha},\tilde{\chi}_{\alpha}}$ is the set of all non-zero vectors of the form 
\begin{equation}
X_{1}\wedge\cdots\wedge X_{d_{\alpha}},\mbox{ where }\forall 1\leq j\leq d_{\alpha}\;\exists\beta\in\Phi_{\alpha}\;\mbox{such that }\; X_{j}\in\frakg_{\beta}\label{eq: V_chi_beta def}
\end{equation}
(see \eqref{eq:chi_alpha} and Remark \ref{rem:chi_alpha_i}). We denote by $\rho_\alpha$ the derivative of $\varrho_{\alpha}$.

\begin{lem}
	\label{lem: w(Delta) reduced}\cite[Prop. 2.62]{key-7} For any reduced
	root $\lambda\in\Phi_{\bbr}$ there exists $\omega\in W\left(\Phi_{\bbr}\right)$
	such that $\lambda\in\omega\left(\Delta_{\bbr}\right)$.
\end{lem}

\begin{lem}
	\label{lem:w(chi)+lambda} Let $\alpha\in\Delta_\bbq$, $\omega\in W\left(\Phi_\bbr\right)$, $\lambda\in\Phi_{\bbr}$, $X_{\lambda}\in\frakg_{\lambda}$, $X_{2\lambda}\in\frakg_{2\lambda}$, and  $v\in V_{\varrho_\alpha,\omega\left(\chi_\alpha\right)}$. Assume  $\omega\left(\chi_\alpha\right)+\lambda\in\Phi_\varrho$. Then  $l=-\left\langle\omega\left(\chi_\alpha\right),\lambda\right\rangle\geq 1$ and  
	\begin{equation}
	\sum_{k\leq\left[\frac{l}{2}\right]}\frac{1}{k!}\cdot\frac{1}{\left(l-2k\right)!}\cdot\rho_{\alpha}\left(X_{2\lambda}\right)^{k}\rho_{\alpha}\left(X_{\lambda}\right)^{l-2k}v=0,
	\label{eq:BC non-zero}
	\end{equation}	
	implies $X_\lambda=X_{2\lambda}=0$ or $v=0$. 
\end{lem}

\begin{proof}
	First note that by Lemma \ref{lem: w(Delta) reduced} we may assume that $\omega$ is the identity.
	
	Since $\chi_\alpha$ is the highest weight, we may deduce that $-\lambda\in\Phi_{\bbr}^+$ and so
	\begin{equation}
	\chi_\alpha-\lambda\notin\Phi_{\varrho}.\label{eq:chi minus}
	\end{equation}
	By \cite[Lem. 10.3]{ha} we have $\left\langle\chi_{\alpha}+\lambda,\lambda\right\rangle\leq 2$, which using \eqref{eq: Weyl group def} and \eqref{eq: l_lambda,mu def} imply $\left\langle\chi_{\alpha},\lambda\right\rangle\leq -1$.
	By Lemma \ref{lem:fundam_rep} we then get 
	\begin{equation}
	-\lambda\mid_{\fraks}\geq 2\alpha.\label{eq:lambda non-redu} 
	\end{equation}
	
	Assume by contradiction that there exist $X_\lambda\in\frakg_{\lambda}$, $X_{2\lambda}\in\frakg_{2\lambda}$, not both
	zero, and a non-zero $v\in V_{\chi_\alpha}$ such that \eqref{eq:BC non-zero} is satisfied. 
	 		
	If $X_{2\lambda}=0$, then by Lemma \ref{lem: =00005BX_lambda ,X_mu=00005D}, \eqref{eq:BC non-zero}, and \eqref{eq:chi minus} we get a contradiction. Thus, we may assume $X_{2\lambda}\neq0$ (and then  $\Phi_{\bbq}$ is non-reduced).
	
	It follows from (\ref{eq: V_chi_beta def}) that for some $Y_{1},\dots Y_{d_\alpha}\in\frakg$
	\[
	v=Y_{1}\wedge\cdots\wedge Y_{d_{\alpha}}
	\]
	Moreover, for any $1\leq i\leq d_{\alpha}$ there exist $\beta_i\in\Phi_{\alpha}$ such that $Y_{i}\in\frakg_{\beta_{i}}$
	(there might be $\beta_{i}=\beta_{j}$ for $i\neq j$). 
	
	For $1\leq i\leq d_{\alpha}$ denote $l_{i}=\left\langle\beta_{i},\lambda\right\rangle$. Then  $l=\sum_{i=1}^{d_{\beta}}l_{i}$ implies
	\begin{equation}
	\begin{array}{l}
	\sum_{k\leq\left[\frac{l}{2}\right]}\frac{1}{k!}\cdot\frac{1}{\left(l-2k\right)!}\cdot\rho_{\alpha}\left(X_{2\lambda}\right)^{k}\rho_{\alpha}\left(X_{\lambda}\right)^{l-2k}v\\
	\quad=\bigwedge_{i=1}^{d_{\alpha}}\sum_{k\leq\left[\frac{l_{i}}{2}\right]}\frac{1}{k!}\cdot\frac{1}{\left(l_{i}-2k\right)!}\cdot\rho_{\alpha}\left(X_{2\lambda}\right)^{k}\rho_{\alpha}\left(X_{\lambda}\right)^{l_{i}-2k}Y_{i}.
	\end{array}\label{eq: W(chi)+lambda_2}
	\end{equation}
	According to Corollary 4 in \cite[\S 7]{key-8}, $v_1,\dots,v_n$ satisfy $v_1 \wedge\cdots\wedge v_n=0$ if and only if they are linearly dependent. Therefore, \eqref{eq:chi minus}, \eqref{eq: W(chi)+lambda_2}, and the linear independence of the weight spaces imply that  there exist $1\leq i\leq d_{\alpha}$ and a non-zero $Y\in\frakg_{\beta_i}$ such that 
	\begin{equation}
	\sum_{k\leq\left[\frac{l_{i}}{2}\right]}\frac{1}{k!}\cdot\frac{1}{\left(l_{i}-2k\right)!}\cdot\rho_\alpha\left(X_{2\lambda}\right)^{k}\rho_\alpha\left(X_{\lambda}\right)^{l_{i}-2k}Y=0.\label{eq: W(mu)+lambda is zero}
	\end{equation}
	
	Since $\Phi_{\bbq}$ is non-reduced, \eqref{eq:Phi_chi} implies that $\beta_{i}\mid_{\fraks}\geq \alpha$. Since $3\alpha\notin\Phi_\bbq$ (a basic property of a root system), using \eqref{eq:lambda non-redu} we may deduce $\beta_i-\lambda\notin\Phi_\bbr$. Thus, \eqref{eq: W(mu)+lambda is zero} is a contradiction to Lemma \ref{lem: =00005BX_lambda ,X_mu=00005D general}. 
\end{proof}

\section{\label{sec:highest weight} Highest Weight Representations}

We preserve the notation of \S \ref{sec: lie algebras}.

Let $\varrho:G\rightarrow\mbox{GL}\left(V\right)$ be an irreducible finite-dimensional $\bbr$-representation of $G$.


There is a direct sum decomposition 
\begin{equation}
V=\bigoplus_{\lambda\in\Phi_{\varrho,\bbr}}V_{\varrho,\lambda}.\label{eq: V decomp}
\end{equation}
For any $\lambda\in\Phi_{\varrho,\bbr}$ let $\varphi_{\lambda}:V\rightarrow V_{\varrho,\lambda}$
be the projection associated with (\ref{eq: V decomp}). 

The goal of this section is
to prove that the norm of the image of a highest weight vector can
be estimated by looking at a small subset of its coefficients.
\begin{thm}
\label{thm:weyl elem norm bound}Let $\varrho:G\rightarrow\mbox{GL}\left(V\right)$
be a strongly rational over $\bbr$ representation with an $\bbr$-highest
weight $\chi$, and $\norm{\cdot}$ be a norm on $V$. Assume either
$\Phi_{\bbr}$ is reduced or $\varrho=\varrho_{\beta}$ for some $\beta\in\Delta_{\bbq}$
(see \S \ref{sub:representations}). Then there exists $c=c\left(\varrho,\norm{\cdot}\right)>0$
such that any $g\in G$ and $v\in V_{\varrho,\chi}$ satisfy 
\[
\norm{\varrho\left(g\right)v}\leq c\cdot\max_{\omega\in W\left(\Phi_{\bbr}\right)}\norm{\varphi_{\omega\left(\chi\right)}\left(\varrho\left(g\right)v\right)}.
\]

\end{thm}
For the rest of this section assume $\varrho$ is strongly rational
over $\bbr$ and either $\Phi_{\bbr}$ is reduced or $\varrho=\varrho_{\beta}$
for some $\beta\in\Delta_{\bbr}$. 

Denote by $\chi$ the highest $\bbr$-weight
for $\varrho$. 
Let $\rho:\frakg\rightarrow\mathfrak{gl}\left(V\right)$ be the derivative
of $\varrho$. 

Let $\norm{\cdot}$ be a norm on $V$. 

\begin{prop}
\label{prop: N decomp}\cite[\S 14.4]{key-1} Assume $\Psi=\left\{ \lambda_{1},\dots,\lambda_{d}\right\} \subset\Phi_{\bbr}^{+}$
is closed under addition, i.e., if $\lambda,\mu\in\Psi$ and $\lambda+\mu\in\Phi_{\bbr}$,
then $\lambda+\mu\in\Psi$. Denote $\mathfrak{n}_{\Psi}=\bigoplus_{i=1}^{d}\frakg_{\lambda_{i}}$,
$N_{\Psi}=\exp\mathfrak{n}_{\Psi}$, and for $1\leq i\leq d$ denote $N_{i}=\exp\frakg_{\lambda_{i}}$ (note that $N_{\Psi}$ does not depend on how the elements of $\Psi$
are ordered). Then 
\[
N_{\Psi}=N_{1}\cdot N_{2}\cdots N_{d}=\left\{ n_{1}\cdot n_{2}\cdots n_{d}\::\: n_{i}\in N_{i}\right\} .
\]
\end{prop}

Let $\mathfrak{n}=\bigoplus_{\lambda\in\Phi_{\bbr}^{+}}\mathfrak{\frakg}_{\lambda}$,
$N=\exp\left(\mathfrak{n}\right)$, and $B=N_{G}\left(N\right)$.
Then $B$ is an $\bbr$-Borel subgroup. 

\begin{lem}
\label{lem: xi+kalpha}Let $\omega\in W\left(\Phi_{\bbr}\right)$,
$n\in N$, $\xi=\omega\left(\chi\right)$, $\alpha\in\Delta_{\bbr}$,
and $v\in V_{\varrho,\xi}$ be non-zero. Assume $\varphi_{\omega_{\alpha}\left(\xi\right)}\left(\varrho\left(n\right)v\right)=0$ (see \eqref{eq: Weyl group def} for the definition of $\omega_{\alpha}$).
Then for any $k\in\bbn$ 
\begin{equation}
\varphi_{\xi+k\alpha}\left(\varrho\left(n\right)v\right)=0.\label{eq: xi+kalpha}
\end{equation}
\end{lem}

\begin{proof}
Assume $\Phi_{\bbr}^{+}=\left\{ \lambda_{1},\dots,\lambda_{d}\right\} $.
Since $\Phi_{\bbr}^{+}$ is closed under addition, we may use Proposition
\ref{prop: N decomp} to write 
\[
n=\exp\left(X_{1}\right)\exp\left(X_{2}\right)\cdots\exp\left(X_{d}\right),\quad X_{i}\in\frakg_{\lambda_{i}}.
\]
 By the definition of the exponential map, for any $X\in\frakg$ we
have 
\begin{equation}
\varrho\left(\exp\left(X\right)\right)=\sum_{k=0}^{\infty}\frac{1}{k!}\rho\left(X\right)^{k}.\label{eq: exponent formula}
\end{equation}
 It follows from (\ref{eq: lambda acts on chi}) that for any $k_{1},\dots,k_{d}\in\bbn$,
$\rho\left(X_{1}\right)^{k_{1}}\cdots\rho\left(X_{d}\right)^{k_{d}}v$
is an $\bbr$-weight vector for $\xi+k_{1}\lambda_{1}+\dots+k_{d}\lambda_{d}$.
Hence, for any $\mu\in\Phi_{\varrho}$ 
\begin{equation}
\varphi_{\mu}\left(\varrho\left(n\right)v\right)=\sum_{\begin{array}{c}
\mu=\xi+k_{1}\lambda_{1}+\dots+k_{d}\lambda_{d}\\
k_{1},\dots,k_{d}\in\bbn
\end{array}}\frac{1}{k_{1}!\cdots k_{d}!}\rho\left(X_{1}\right)^{k_{1}}\cdots\rho\left(X_{d}\right)^{k_{d}}v.\label{eq: mu coeff}
\end{equation}
 Assume $\alpha=\lambda_{i}$, $2\alpha=\lambda_{j}$, and without
loss of generality assume $j<i$. Since $\alpha$ is an $\bbr$-simple
root, (\ref{eq: mu coeff}) implies that for any $l\in\bbn$,
\begin{equation}
\varphi_{\xi+l\alpha}\left(\varrho\left(n\right)v\right)=\sum_{k\leq\left[\frac{l}{2}\right]}\frac{1}{k!}\frac{1}{\left(l-2k\right)!}\rho\left(X_{j}\right)^{k}\rho\left(X_{i}\right)^{l-2k}v\label{eq:xi+kalpha exponent}
\end{equation}
(if $2\alpha\notin\Phi_{\bbr}$, then $X_{j}=0$). 

If $\xi+\alpha\notin\Phi_{\varrho}$, then $\rho\left(X_{i}\right)v=0$,
and (\ref{eq: xi+kalpha}) can be deduced from (\ref{eq:xi+kalpha exponent}).
Assume otherwise. 

If $\Phi_\bbq$ is non-reduced, then equation (\ref{eq:xi+kalpha exponent}) and Lemma \ref{lem:w(chi)+lambda}
imply $X_{i}=X_{j}=0$. 

Assume $\Phi_\bbq$ is reduced. Then $X_j=0$. The maximality of $\chi$ the assumption $\xi+\alpha\in\Phi_{\varrho}$ and Lemma \ref{lem: w(Delta) reduced} imply that $\xi-\alpha\notin\Phi_{\varrho}$. Then, according to Lemma \ref{lem: =00005BX_lambda ,X_mu=00005D}, $X_{i}$. 

In both cases $X_{i}=X_{j}=0$, and so (\ref{eq: xi+kalpha}) can be deduce
from (\ref{eq:xi+kalpha exponent}).
\end{proof}

According to \cite[\S 11.19]{key-1} the Weyl group satisfies $W\left(\Phi_{\bbr}\right)=N_{G}\left(T\right)/Z_{G}\left(T\right)$.
For $\omega\in W\left(\Phi_{\bbr}\right)$ let $\bar{\omega}$ be
a representative of $\omega$ in $N_{G}\left(T\right)$. For
any $\lambda\in\Phi_{\varrho}$ and $\omega\in W\left(\Phi_{\bbr}\right)$
we have 
\begin{equation}
\rho\left(\bar{\omega}\right)V_{\varrho,\lambda}=V_{\varrho,\omega\left(\lambda\right)}.\label{eq: weyl act weight}
\end{equation}

\begin{thm}[Bruhat decomposition]
\label{thm: bruhat decomp}\cite[\S IX.1]{key-2} We have 
\[
G=\biguplus_{\omega\in W\left(\Phi_{\bbr}\right)}N\bar{\omega}B,
\]
where $\biguplus$ denotes a disjoint union. 
\end{thm}
Fix an order $\Phi_{\bbr}^{+}=\left\{ \lambda_{1},\dots,\lambda_{d}\right\} $
so that 
\begin{equation}
\mbox{any }1\leq i<j\leq d\mbox{ satisfies }\lambda_{i}\ngeq\lambda_{j}.\label{eq: property of Phi'}
\end{equation}
For $\omega\in W\left(\Phi_{\bbr}\right)$ let 
\[
\left\{ \lambda\in\Phi_{\bbr}^{+}\::\:\omega\left(\chi\right)+\lambda\in\Phi_{\varrho}\right\} =\left\{ \lambda_{1}^{\omega},\dots,\lambda_{k\left(\omega\right)}^{\omega}\right\} ,
\]
where the indexation is the one induced from the indexation on $\Phi_{\bbr}^{+}$,
and let 
\[
U_{\omega}=\exp\left(\frakg_{\lambda_{1}^{\omega}}\right)\cdot\exp\left(\frakg_{\lambda_{2}^{\omega}}\right)\cdots\exp\left(\frakg_{\lambda_{k\left(\omega\right)}^{\omega}}\right).
\]

\begin{prop}
\label{prop:f(x)=00003D0 =00003D> x=00003D0}Let $\omega\in W\left(\Phi_{\bbr}\right)$,
$n\in U_{\omega}$, $v\in V_{\varrho,\omega\left(\chi\right)}$. If
$n$ is not the identity and $v$ is non-zero, then there exists $\omega^{\prime}\in W\left(\Phi_{\bbr}\right)$
such that $\omega^{\prime}\left(\chi\right)\neq\omega\left(\chi\right)$
and $\varphi_{\omega^{\prime}\left(\chi\right)}\left(\varrho\left(n\right)v\right)\neq0$.\end{prop}
\begin{proof}
Arguing by contradiction, assume there exist $\omega\in W\left(\Phi_{\bbr}\right)$,
$1\neq n\in U_{\omega}$, and $0\neq v\in V_{\varrho,\omega\left(\chi\right)}$,
such that any $\omega^{\prime}\in W\left(\Phi_{\bbr}\right)$, $\omega^{\prime}\left(\chi\right)\neq\omega\left(\chi\right)$,
satisfies $\varphi_{\omega^{\prime}\left(\chi\right)}\left(\varrho\left(n\right)v\right)=0$.

Denote $\xi=\omega\left(\chi\right)$. Let $\lambda\in\Phi_{\bbr}^{+}$
be a reduced $\bbr$-root which satisfies $\xi+\lambda\in\Phi_{\varrho,\bbr}$.
According to Lemma \ref{lem: w(Delta) reduced} there exists
$s_{1}\in W\left(\Phi_{\bbr}\right)$ so that $\lambda\in s_{1}\left(\Delta_{\bbr}\right)$.
Then $s_{1}\left(\chi\right)$ is the highest weight according to
the order defined by $s_{1}\left(\Delta_{\bbr}\right)$. Let $v_{1}=\varrho\left(\bar{s}_{1}\bar{\omega}^{-1}\right)v\in V_{\varrho,s_{1}\left(\chi\right)}$
and $g=n\bar{\omega}\bar{s}_{1}^{-1}$, then 
\begin{equation}
\varrho\left(n\right)v=\varrho\left(g\right)v_{1}.\label{eq:n->g}
\end{equation}
Replace $\Delta_{\bbr}$ with $s_{1}\left(\Delta_{\bbr}\right)$.
By Theorem \ref{thm: bruhat decomp}, $g=n_{2}\bar{s}_{2}b$ where
$n_{2}\in N$, $s_{2}\in W\left(\Phi_{\bbr}\right)$, $b\in B$
(note that here $\Phi_{\bbr}^{+}$ is defined using $s_{1}\left(\Delta_{\bbr}\right)$).
Since $s_{1}\left(\chi\right)$ is the highest weight, according to
the new order, $\varrho\left(b\right)v_{1}\in V_{\varrho,s_{1}\left(\chi\right)}$.
Then, there exists a non-zero vector $v_{2}\in V_{\varrho,s_{2}s_{1}\left(\chi\right)}$
such that 
\begin{equation}
\varrho\left(g\right)v_{1}=\varrho\left(n_{2}\right)v_{2}.\label{eq: g->n2}
\end{equation}
It follows from (\ref{eq: exponent formula}) and $n_{2}\in N$
that $\varphi_{s_{2}s_{1}\left(\chi\right)}\left(\varrho\left(n_{2}\right)v_{2}\right)=v_{2}$.
Thus, by (\ref{eq:n->g}) and (\ref{eq: g->n2}), $\varphi_{s_{2}s_{1}\left(\chi\right)}\left(\varrho\left(n\right)v\right)=v_{2}$.
Since $\varphi_{s\left(\chi\right)}\left(\varrho\left(n\right)v\right)=0$
if $s\left(\chi\right)\neq\xi$, and $v_{2}$ is non-zero, we can
deduce $s_{2}s_{1}\left(\chi\right)=\xi$. Since $\xi+\lambda\in\Phi_{\varrho,\bbr}$,
by Lemma \ref{lem:w(chi)+lambda}$\left(i\right)$ and the assumption
we have $\varphi_{\omega_{\lambda}\left(\xi\right)}\left(\varrho\left(n_{2}\right)v_{2}\right)=0$.
Then it follows from Lemma \ref{lem: xi+kalpha} and equations (\ref{eq:n->g}),
(\ref{eq: g->n2}) that $\varphi_{\xi+\lambda}\left(\varrho\left(n\right)v\right)=\varphi_{\xi+2\lambda}\left(\varrho\left(n\right)v\right)=0$.
Hence, any positive real root $\lambda$ satisfies 
\begin{equation}
\varphi_{\xi+\lambda}\left(\varrho\left(n\right)v\right)=0.\label{eq: xi+lambda =00003D0}
\end{equation}

By the definition of $U_{\omega}$ 
\[
n=\exp\left(X_{1}\right)\cdots\exp\left(X_{k\left(\omega\right)}\right),\quad X_{i}\in\frakg_{\lambda_{i}^{\omega}}.
\]
We will now show by induction on $i$ that $X_{i}=0$. Let $1\leq i\leq k\left(\omega\right)$
and assume $X_{1}=\dots=X_{i-1}=0$. As in the proof of Lemma \ref{lem: xi+kalpha},
it follows from (\ref{eq: exponent formula}), (\ref{eq: property of Phi'}),
and the induction assumption that 
\[
\varphi_{\xi+\lambda_{i}^{\omega}}\left(\varrho\left(n\right)v\right)=\rho\left(X_{i}\right)v.
\]
Since $v\neq0$, by (\ref{eq: xi+lambda =00003D0}) and Lemma \ref{lem: =00005BX_lambda ,X_mu=00005D}
it follows that $X_{i}=0$. Thus $n$ is the identity element, a contradiction. 
\end{proof}
Let $k_{1},k_{2}\in\bbn$, $l=\left(l_{1},\dots,l_{k_{1}}\right)\in\bbr_{+}^{k_{1}}$.
We say $f:\bbr^{k_{1}}\rightarrow\mathbb{\bbr}^{k_{2}}$ is \textbf{$l$-homogeneous}
if for any $x_{1},\dots,x_{k_{1}}\in\bbr$, $a\geq0$,
\[
f\left(a^{l_{1}}x_{1},\dots,a^{l_{k_{1}}}x_{k_{1}}\right)=af\left(x_{1},\dots,x_{k_{1}}\right).
\]
 Let $\norm{\cdot}_{l}$ be the following quasi-norm on $\bbr^{k_{1}}$
\[
\norm x_{l}=\max_{1\leq i\leq k_{1}}\left|x_{i}\right|^{\nicefrac{1}{l_{i}}},\quad x=\left(x_{1},\dots,x_{k_{1}}\right)\in\bbr^{k_{1}}.
\]

\begin{lem}
\label{lem: hom upper and lower }Let $k_{1},k_{2}\in\bbn$, $l\in\bbr_{+}^{k_{1}}$,
and assume a norm on $\bbr^{k_{2}}$. Let $f:\bbr^{k_{1}}\rightarrow\bbr^{k_{2}}$
be a continuous $l$-homogeneous function. Then there exists $c_{1}>0$
such that for any $x\in\bbr^{k_{1}}$, 
\[
c_{1}\cdot\norm{f\left(x\right)}\leq\norm x_{l}.
\]
If $f\left(x\right)=0$, $x\in\bbr^{k_{1}}$, only when $x=0$, then
there exists $c_{2}>0$ such that for any $x\in\bbr^{k_{1}}$, 
\[
\norm x_{l}\leq c_{2}\cdot\norm{f\left(x\right)}.
\]

\end{lem}
The proof of Lemma \ref{lem: hom upper and lower } is left as an
exercise for the reader.

\begin{lem}
	\label{lem: w(Delta)}\cite[Thm. 2.63]{key-7} The $\bbr$-Weyl group
	acts simply transitively on $\bbr$-simple systems. That is, if $\Delta$
	and $\Delta^{\prime}$ are two $\bbr$-simple systems for $\Phi_{\bbr}$,
	then there exists one and only one element $\omega\in W\left(\Phi_{\bbr}\right)$
	such that $\omega\left(\Delta\right)=\Delta^{\prime}$.
\end{lem}

\begin{prop}
\label{prop: weyl bound implies bound}Let $\omega\in W\left(\Phi_{\bbr}\right)$.
For any $c_{1}\geq1$ there exists $c_{2}\geq1$ such that if $n\in U_{\omega}$,
$0\neq v\in V_{\varrho,\omega\left(\chi\right)}$ satisfy 
\begin{equation}
\max_{s\in W\left(\Phi_{\bbr}\right)}\norm{\varphi_{s\left(\chi\right)}\left(\varrho\left(n\right)v\right)}\leq c_{1}\cdot\norm v,\label{eq: n bound}
\end{equation}
then $\norm{\varrho\left(n\right)v}\leq c_{2}\cdot\norm v$. \end{prop}
\begin{proof}
Let $\left\{ E_{1},\dots,E_{k}\right\} $ be a basis for $\mathfrak{n}_{\bbr}$
which satisfies the following. For any $1\leq i\leq d$ there exists
$1\leq j\left(i\right)\leq k$, $j\left(0\right)=1$, such that $\left\{ E_{j\left(i-1\right)},\dots,E_{j\left(i\right)}\right\} $
is a basis for $\frakg_{\lambda_{i}}$. Let $e\in V_{\varrho,\omega\left(\chi\right)}$
be of norm one. By the definition of strongly rational over $\bbr$
representations and (\ref{eq: weyl act weight}), the vector $e$
spans $V_{\varrho,\omega\left(\chi\right)}$. 

For $1\leq i\leq d$ denote $l_{i}=\sum_{\alpha\in\Delta_{\bbr}}m_{\alpha}\left(\lambda_{i}\right)$,
and for $\mu\in\Phi_{\varrho}$ denote 
\[
l_{\mu}=\begin{cases}
1 & \mbox{if }\mu\leq\omega\left(\chi\right)\\
\sum_{\alpha\in\Delta_{\bbr}}m_{\alpha}\left(\mu-\omega\left(\chi\right)\right) & \mbox{otherwise}
\end{cases}
\]
where for an $\bbr$-weight $\lambda$, $m_{\alpha}\left(\lambda\right)$
is defined as in (\ref{eq: lambda=00003Dsum n(lambda) mu}). For $\mu\in\Phi_{\varrho}$
define $f_{\mu}:\bbr^{k}\rightarrow\bbr$ by 
\begin{eqnarray*}
f_{\mu}\left(x_{1},\dots,x_{k}\right) & = & \biggl\Vert\varphi_{\mu}\Bigl(\varrho\Bigl(\exp\left(x_{1}E_{1}+\dots+x_{j\left(1\right)}E_{j\left(1\right)}\right)\circ\\
 &  & \cdots\circ\exp\left(x_{j\left(d-1\right)}E_{j\left(d-1\right)}+\dots+x_{k}E_{k}\right)\Bigl)e\Bigl)\biggl\Vert^{\nicefrac{1}{l_{\mu}}}.
\end{eqnarray*}
If $\mu\geq\omega\left(\chi\right)$, it follows from (\ref{eq: exponent formula})
that 
\begin{eqnarray*}
f_{\mu}\left(x_{1},\dots,x_{k}\right) & = & \Biggl\Vert\sum\frac{1}{b_{1}!\cdots b_{m}!}\rho\left(x_{j\left(i_{1}-1\right)}E_{j\left(i_{1}-1\right)}+\dots+x_{j\left(i_{1}\right)}E_{j\left(i_{1}\right)}\right)^{b_{1}}\circ\\
 &  & \cdots\circ\rho\left(x_{j\left(i_{k}-1\right)}E_{j\left(i_{k}-1\right)}+\dots+x_{j\left(i_{k}\right)}E_{j\left(i_{m}\right)}\right)^{b_{m}}e\Biggl\Vert^{\nicefrac{1}{l_{\mu}}},
\end{eqnarray*}
where the sum is over all $i_{1}<\cdots<i_{m}$ such that $b_{1}\lambda_{i_{1}}+\dots+b_{m}\lambda_{i_{m}}=\mu-\omega\left(\chi\right)$.
Since $\mu<\omega\left(\chi\right)$ implies $f_{\mu}\left(x\right)=0$,
we may deduce each $f_{\mu}$ is a continuous $l=\left(l_{1},\dots,l_{k}\right)$-homogeneous
function for any $\mu\in\Phi_{\varrho,\bbr}$. 

Let
\[
\left\{ \mu_{1},\dots,\mu_{k_{1}}\right\} =\left\{ s\left(\chi\right)\::\: s\in W\left(\Phi_{\bbr}\right),\: s\left(\chi\right)\neq\omega\left(\chi\right)\right\} ,
\]
$f:\bbr^{k}\rightarrow\bbr^{k_{1}}$ be defined by $f=\left(f_{\mu_{1}},\dots,f_{\mu_{k_{1}}}\right)$%
. Then $f$ is a continuous $l$-homogeneous function, and, by Proposition
\ref{prop:f(x)=00003D0 =00003D> x=00003D0}, $f\left(x\right)=0$
implies $x=0$. Hence, according to Lemma \ref{lem: hom upper and lower },
there exists $c_{2}^{\prime}>0$ such that for any $x\in\bbr^{k}$
\begin{equation}
\norm x_{l}\leq c_{2}^{\prime}\cdot\norm{f\left(x\right)}.\label{eq: x<f(x)}
\end{equation}

There exists $a=\left(a_{1},\dots,a_{k}\right)\in\bbr^{k}$ such that
\[
n=\exp\left(a_{1}E_{1}+\dots+a_{j\left(1\right)}E_{j\left(1\right)}\right)\cdots\exp\left(a_{j\left(d-1\right)}E_{j\left(d-1\right)}+\dots+a_{k}E_{k}\right).
\]
Since $e$ spans $V_{\varrho,\omega\left(\chi\right)}$ and is of
norm one, and by the linearity of the representation, for any $\mu\in\Phi_{\varrho}$
\begin{equation}
\norm{\varphi_{\mu}\left(\varrho\left(n\right)v\right)}=\norm v\cdot\norm{f_{\mu}\left(a\right)}^{l_{\mu}}.\label{eq: phi vs f}
\end{equation}
Then by (\ref{eq: x<f(x)}) and the assumption 
\begin{equation}
\norm a_{l}\leq c_{2}^{\prime}\cdot\norm{f\left(a\right)}\leq c_{2}^{\prime}\cdot\frac{1}{\norm v}\cdot c_{1}\norm v=c_{2}^{\prime}\cdot c_{1}.\label{eq:a bound}
\end{equation}

Assume $\Phi_{\varrho}=\left\{ \mu_{1},\dots,\mu_{k_{2}}\right\} $.
We may apply Lemma \ref{lem: hom upper and lower } to $f^{\prime}=\left(f_{\mu_{1}},\dots,f_{\mu_{k_{2}}}\right)$
which is also a continuous $l$-homogeneous function, in order to
deduce there exists $c_{1}^{\prime}>0$ such that for any $\mu\in\Phi_{\varrho}$,
$x\in\bbr^{k}$, 
\begin{equation}
\norm{f_{\mu}\left(x\right)}\leq c_{1}^{\prime}\cdot\norm x_{l}.\label{eq:f bound}
\end{equation}
Without loss of generality assume that for any $u\in V$
\begin{equation}
\norm u=\max_{\mu\in\Phi_{\varrho,\bbr}}\norm{\varphi_{\mu}\left(u\right)}.\label{eq: norm bound}
\end{equation}

Denote $c_{2}=\max_{\mu\in\Phi_{\varrho}}\left(c_{1}^{\prime}\cdot c_{2}^{\prime}\cdot c_{1}\right)^{l_{\mu}}$.
Then, by (\ref{eq: phi vs f}), (\ref{eq:a bound}), (\ref{eq:f bound}),
and (\ref{eq: norm bound}) 
\begin{eqnarray*}
\norm{\varrho\left(n\right)v} & = & \max_{\mu\in\Phi_{\varrho,\bbr}}\norm{\varphi_{\mu}\left(\varrho\left(n\right)v\right)}\\
 & = & \norm v\cdot\max_{\mu\in\Phi_{\varrho}}\norm{f_{\mu}\left(a\right)}^{l_{\mu}}\\
 & \leq & c_{2}\cdot\norm v.
\end{eqnarray*}

\end{proof}

\begin{proof}[Proof of Theorem \ref{thm:weyl elem norm bound}]
Denote by $c$ the constant $c_{2}$ which satisfies the conclusion
of Proposition \ref{prop: weyl bound implies bound} for $c_{1}=1$. 

Let $s\in W\left(\Phi_{\bbr}\right)$ satisfy 
\[
\norm{\varphi_{s\left(\chi\right)}\left(\varrho\left(g\right)v\right)}=\max_{\omega\in W\left(\Phi_{\bbr}\right)}\norm{\varphi_{\omega\left(\chi\right)}\left(\varrho\left(g\right)v\right)}.
\]
By Theorem \ref{lem: w(Delta)} there exists $s_{1}\in W\left(\Phi_{\bbr}\right)$
such that $s_{1}\left(\Delta_{\bbr}\right)=-s\left(\Delta_{\bbr}\right)$.
According to the order defined using $s_{1}\left(\Delta_{\bbr}\right)$,
the $\bbr$-highest weight for $\varrho$ is $s_{1}\left(\chi\right)$.
Moreover, $s\left(\chi\right)$ is the $\bbr$-lowest weight for $\varrho$
according to this order, i.e. every $\bbr$-weight of $\varrho$ is
of the form $s\left(\chi\right)+\sum_{\alpha\in\Delta_{\bbr}}m_{\alpha}\alpha,$
with non-negative integers $m_{\alpha}$. Let $v^{\prime}=\varrho\left(\bar{s}_{1}\right)v$.
Then $v^{\prime}\in V_{\varrho,s_{1}\left(\chi\right)}$ and 
\[
\varrho\left(g\right)v=\varrho\left(g\bar{s}_{1}^{-1}\right)v^{\prime}.
\]
 Hence, by replacing $\Delta_{\bbr}$ with $s_{1}\left(\Delta_{\bbr}\right)$,
$\chi$ with $s_{1}\left(\chi\right)$, $v$ with $v^{\prime}$, and
$g$ with $g\bar{s}_{1}^{-1}$, we may assume 
\begin{equation}
\norm{\varphi_{\xi}\left(\varrho\left(g\right)v\right)}=\max_{\omega\in W\left(\Phi_{\bbr}\right)}\norm{\varphi_{\omega\left(\chi\right)}\left(\varrho\left(g\right)v\right)},\label{eq: xi is highest}
\end{equation}
where $\xi$ is the $\bbr$-lowest weight for $\varrho$. 

According to Theorem \ref{thm: bruhat decomp}, $g=n\bar{\omega}b$
where $n\in N$, $\omega\in W\left(\Phi_{\bbr}\right)$, $b\in B$.
Since $\chi$ is the highest weight, there exists $v^{\prime}\in V_{\varrho,\omega\left(\chi\right)}$
so that 
\begin{equation}
\varrho\left(g\right)v=\varrho\left(n\right)v^{\prime}.\label{eq: g equal n}
\end{equation}
Since $\xi$ is the lowest weight, $n\in N$, and $\varphi_{\xi}\left(\varrho\left(g\right)v\right)\neq0$,
by (\ref{eq: exponent formula}) we may deduce $\omega\left(\chi\right)=\xi$. 

Order $\Phi_{\bbr}^{+}=\left\{ \mu_{1},\dots,\mu_{d}\right\} $ so
that 
\[
\left\{ \mu_{1},\dots,\mu_{k\left(\omega\right)}\right\} =\left\{ \lambda_{1}^{\omega},\dots,\lambda_{k\left(\omega\right)}^{\omega}\right\} .
\]
Using Proposition \ref{prop: N decomp} with $\Psi=\Phi_{\bbr}^{+}$
we may write 
\[
n=\exp\left(X_{1}\right)\cdots\exp\left(X_{d}\right),\quad X_{i}\in\frakg_{\mu_{i}}.
\]
For $k\left(\omega\right)+1\leq i\leq d$, we have $\xi+\lambda_{i}\notin\Phi_{\varrho}$.
Then by (\ref{eq: exponent formula}) we have $\exp\left(X_{i}\right)v=v$.
Let $n^{\prime}=\exp\left(X_{1}\right)\cdots\exp\left(X_{k\left(\omega\right)}\right)$,
then 
\begin{equation}
\varrho\left(n\right)v^{\prime}=\varrho\left(n^{\prime}\right)v^{\prime}.\label{eq: n equal n'}
\end{equation}
By (\ref{eq: xi is highest}), (\ref{eq: g equal n}), and (\ref{eq: n equal n'})
we have 
\[
\norm{\varphi_{\xi}\left(\varrho\left(n^{\prime}\right)v^{\prime}\right)}=\max_{\omega\in W\left(\Phi_{\bbr}\right)}\norm{\varphi_{\omega\left(\chi\right)}\left(\varrho\left(n^{\prime}\right)v^{\prime}\right)}.
\]
Since $\varphi_{\xi}\left(\varrho\left(n^{\prime}\right)v^{\prime}\right)=v^{\prime}$,
Proposition \ref{prop: weyl bound implies bound} implies 
\begin{equation}
\norm{\varrho\left(n^{\prime}\right)v^{\prime}}\leq c\cdot\norm{\varphi_{\xi}\left(\varrho\left(n^{\prime}\right)v^{\prime}\right)}.\label{eq: almos finish}
\end{equation}
Now, (\ref{eq: g equal n}), (\ref{eq: n equal n'}), and (\ref{eq: almos finish})
imply the conclusion of the theorem. 
\end{proof}
The following example shows that the reals in Theorem \ref{thm:weyl elem norm bound}
cannot be replaced with the rationals, and hence that the assumption
on the real rank of $G$ in the proof of Theorem \ref{thm: all obvoius}
is critical. 
\begin{example}
Let 
\[
\underline{1}=\begin{pmatrix}1 & 0\\
0 & 1
\end{pmatrix},\quad i=\begin{pmatrix}0 & -1\\
1 & 0
\end{pmatrix},\quad j=\begin{pmatrix}\sqrt{2} & 0\\
0 & -\sqrt{2}
\end{pmatrix},\quad k=\begin{pmatrix}0 & \sqrt{2}\\
\sqrt{2} & 0
\end{pmatrix}.
\]
One can check by direct computation that $D=\bbq\cdot1+\bbq\cdot i+\bbq\cdot j+\bbq\cdot k$
is a central division algebra over $\bbq$%
. Since $D\otimes_{\bbq}\bbr\cong M\left(2,\bbr\right)$, the matrix
algebra $M\left(3,D\right)$ is naturally imbedded in $M\left(6,\bbr\right)$.
Denote by $G$ the real algebraic group defined over $\bbq$ by 
\[
G\left(\bbq\right)=\mbox{SL}_{3}\left(D\right)=\left\{ g\in M\left(3,D\right)\::\:\det\left(g\right)=1\right\} .
\]
 Then $G=\mbox{SL}_{6}\left(\bbr\right)$ and 
\[
\frakg_{\bbq}=\left\{ X\in M\left(3,D\right)\::\:\mbox{trace}\left(X\right)=0\right\} .
\]
Then 
\[
S=\left\{ \mbox{diag}\left(s_{1},s_{1},s_{2},s_{2},s_{3},s_{3}\right)\::\: s_{1},s_{2},s_{3}\in\bbr,\; s_{1}\cdot s_{2}\cdot s_{3}=1\right\} 
\]
is a maximal $\bbq$-split torus in $G$. Let $\varrho:G\rightarrow\mbox{GL}\left(\bigwedge^{8}\frakg\right)$
be the eighth exterior power of the adjoint representation. Then $\varrho$
is a strongly rational over $\bbq$-representation. Denote $\underline{0}=\begin{pmatrix}0 & 0\\
0 & 0
\end{pmatrix}$. There exists an order on the rational simple system so that $12s_{1}$
is the $\bbq$-highest weight and 
\begin{eqnarray*}
v & = & \begin{pmatrix}\underline{0} & \underline{0} & \underline{1}\\
\underline{0} & \underline{0} & \underline{0}\\
\underline{0} & \underline{0} & \underline{0}
\end{pmatrix}\wedge\begin{pmatrix}\underline{0} & \underline{0} & i\\
\underline{0} & \underline{0} & \underline{0}\\
\underline{0} & \underline{0} & \underline{0}
\end{pmatrix}\wedge\begin{pmatrix}\underline{0} & \underline{0} & j\\
\underline{0} & \underline{0} & \underline{0}\\
\underline{0} & \underline{0} & \underline{0}
\end{pmatrix}\wedge\begin{pmatrix}\underline{0} & \underline{0} & k\\
\underline{0} & \underline{0} & \underline{0}\\
\underline{0} & \underline{0} & \underline{0}
\end{pmatrix}\\
 &  & \wedge\begin{pmatrix}\underline{0} & \underline{1} & \underline{0}\\
\underline{0} & \underline{0} & \underline{0}\\
\underline{0} & \underline{0} & \underline{0}
\end{pmatrix}\wedge\begin{pmatrix}\underline{0} & i & \underline{0}\\
\underline{0} & \underline{0} & \underline{0}\\
\underline{0} & \underline{0} & \underline{0}
\end{pmatrix}\wedge\begin{pmatrix}\underline{0} & j & \underline{0}\\
\underline{0} & \underline{0} & \underline{0}\\
\underline{0} & \underline{0} & \underline{0}
\end{pmatrix}\wedge\begin{pmatrix}\underline{0} & k & \underline{0}\\
\underline{0} & \underline{0} & \underline{0}\\
\underline{0} & \underline{0} & \underline{0}
\end{pmatrix},
\end{eqnarray*}
is a $\bbq$-highest weight vector. Let 
\[
g=\begin{pmatrix}1 & 0 & 0 & 0 & 0 & 0\\
0 & 0 & 1 & 0 & 0 & 0\\
0 & 1 & 0 & 0 & 0 & 0\\
0 & 0 & 0 & 1 & 0 & 0\\
0 & 0 & 0 & 0 & 1 & 0\\
0 & 0 & 0 & 0 & 0 & 1
\end{pmatrix}.
\]
Then it can be directly checked that $\varrho\left(g\right)v$ is
a $\bbq$-weight vector for $-6s_{3}$. Since 
\[
\left\{ \omega\left(12s_{1}\right)\::\:\omega\in W\left(\Phi_{\bbq}\right)\right\} =\left\{ 12s_{1},\;12s_{2},\;12s_{3}\right\} ,
\]
we get that 
\[
\max_{\omega\in W\left(\Phi_{\bbq}\right)}\norm{\varphi_{\omega\left(\chi\right)}\left(\varrho\left(g\right)v\right)}=0
\]
even though $\varrho\left(g\right)v$ is a non-zero vector. 
\end{example}

\section{\textup{\label{sec:Compactness-criterion}}Compactness Criterion}

Recall that $\Delta_{\bbq}=\left\{\alpha_{1},\dots,\alpha_{r}\right\}$ is a $\bbq$-simple system for the rational root system $\Phi_{\bbq}$. 
 
Let $\calb$ be a rational basis for $\frakg_{\bbz}$ such that $\calb_{\beta}\subset\frakg_{\beta}$ for any $\beta\in\Phi_{\bbq}\cup\left\{0\right\}$ and $\calb=\biguplus_{\beta\in\Phi_{\bbq}\cup\left\{0\right\}}\calb_{\beta}$. Denote the $\bbz$-span of $\calb$ by $\frakg_{\bbz}$. 

For $1\leq i\leq r$ let 
\[
\fraku_{i}=\bigoplus_{\beta\in\Phi_{\bbq},\:\beta\geq\alpha_{i}}\frakg_{\beta},\quad P_{i}=N_{G}\left(\fraku_{i}\right),
\]
and $\calb_{i}=\biguplus_{\beta\geq\alpha_{i}}\calb_{\beta}$. 
Then, $P_{1},\dots,P_{r}$ are the maximal $\bbq$-parabolic
subgroups of $G$ containing $B$ (see \S \ref{sec:highest weight} for the definition of $B$), $\fraku_{1},\dots,\fraku_{r}$ are the Lie algebras of their unipotent radicals, and $\calb_{1},\dots,\calb_{r}$ are bases for $\fraku_{1},\dots,\fraku_{r}$.

\begin{defn}\label{def:active}
	For a neighborhood $W$ of zero in $\frakg$, $g\in G$, and $1\leq i\leq r$, we say that $q\in G\left(\bbq\right)$ is \textbf{$\left(W,i\right)$-active} for $g$ if\[
	\mbox{Ad}\left(gq\right)\calb_i\subset\left(\mbox{Ad}(g)\frakg_{\bbz}\cap W\right).
	\] 
\end{defn}

The following useful criterion is similar to the one proved in \cite[Prop. 3.5]{key-3} and can be deduced from its proof.
\begin{prop}[Compactness criterion]
	\label{Compactness criterion}
	There exists a finite subset $C_0\subset G\left(\bbq\right)$ which satisfies the following.
	A subset $A\subset G/\Gamma$ is unbounded if and only if for any neighborhood $W$ of zero in $\frakg$ there is $g\in G$, $\pi\left(g\right)\in A$, $1\leq i\leq r$, and $q\in\Gamma C_0$ which is $\left(W,i\right)$-active for $g$.
\end{prop}
\begin{rem}\label{rem:compact}
Since $C_0$ is a finite subset of $G\left(\bbq\right)$, by changing $\calb$ we can assume $C_0$ only contains the identity. 
\end{rem}

Let $\fraku_{0}=\bigoplus_{\beta\in\Phi_\bbq,\:\beta>0}\frakg_{\beta}$. 
A subset of $\frakg$ is called \textbf{horospherical} if it is contained in a $\bbq$-subalgebra conjugate to $\mathfrak{n}_0$.

\begin{prop}
	\label{prop:W unipotent}\cite[Prop. 3.3]{key-3}
	There exists a neighborhood $W_0$ of zero in $\frakg$ such that for any $g\in G$, the span of $\mbox{Ad}\left(g\right)\frakg_\bbz\cap W_0$ is horospherical.
\end{prop}

\begin{prop}
\label{prop: conj rad equal}\cite[Prop. 3.5]{key-18} Let $1\leq i\leq r$.
If $\mathfrak{v}\subset\frakg$ is conjugate to $\fraku_{i}$ and $\mathfrak{v}\subset\mathfrak{n}_{0}$,
then $\mathfrak{v}=\fraku_{i}$. 
\end{prop}

\begin{prop}
\label{prop: unique}There exists a neighborhood $W_0$ of zero in $\frakg$ such that any neighborhood of zero $W\subset W_0$ satisfies the following.
Let $A\subset G$ be a connected set and $1\leq i\leq r$. 
Assume that for each $g\in A$ there exists $\gamma_g\in\Gamma$ which is $\left(W,i\right)$-active for $g\in G$. 
Then for any $g,h\in A$ we have $\gamma_{g}^{-1}\gamma_h\in P_i$.
\end{prop}

\begin{proof}
	Let $W_0$ be as in Proposition \ref{prop:W unipotent} and $W\subset W_0$. 
	
	Let $g\in A$. 
	By the continuity of the adjoint representation, there exists a neighborhood $H$ of $g$ such that for any $h\in H$\[
	\mbox{Ad}\left(\gamma_g\right)\calb_i\subset\mbox{span}\left(\frakg_\bbz\cap \mbox{Ad}\left(h^{-1}\right)W\right).
	\]
	This implies that  \[\mbox{Ad}\left(h\gamma_g\right)\fraku_{i},\mbox{Ad}\left(h\gamma_h\right)\fraku_{i}\subset\mbox{span}\left(\mbox{Ad}\left(h\right)\frakg_\bbz\cap W\right).\] By Proposition \ref{prop:W unipotent}, the span of $\mbox{Ad}\left(h\right)\frakg_\bbz\cap W$ is contained in a conjugate of $\mathfrak{n}_0$.
	Then, Proposition \ref{prop: unique} implies $\mbox{Ad}\left(h\gamma_g\right)\fraku_{i}=\mbox{Ad}\left(h\gamma_h\right)\fraku_{i}$. Therefore, $\gamma_{g}^{-1}\gamma_h$ is in the normalizer of $\fraku_{i}$ which is $P_i$. 
	Since $A$ is connected, the conclusion of the proposition follows. 
\end{proof}

We use the notation of \S \ref{sub:representations}. To simplify some of them, in this section we denote $d_{\alpha_i}$ by $d_{i}$ and $\varrho_{\alpha_i}$ by $\varrho_{i}$, for any $1\leq i\leq r$.
Note that if $\Phi_{\bbq}$ is not reduced then $\fraku_{\alpha_{1}},\dots,\fraku_{\alpha_{r}}$ are proper subalgebras of $\fraku_{1},\dots,\fraku_{r}$. 

For any $1\leq i\leq r$ let $\Phi_{\alpha_i}^{c}$ be the set of all real roots which lie in the span of $\Delta_\bbq\setminus\left\{\alpha_i\right\}$. Then for any $1\leq i\leq r$, $P_{i}$ is the subgroup generated by $Z_{G}\left(S\right)$ and $\exp\left(\frakg_{\beta}\right)$ for any $\beta\in\Phi_{\alpha_i}^{c}$ (see \cite[21.11]{key-1}). By the definition of $\Phi_{\alpha_i}$ and $\Phi_{\alpha_i}^{c}$, $1\leq i\leq r$, it is easy to see that for any $\beta\in\Phi_{\alpha_i}$, $\lambda\in\Phi_{\alpha_i}^{c}$, $c_1>0$, and $c_2\geq0$ we have $c_1\beta+c_2\lambda\in\Phi_{\alpha_i}$. Hence, for any $1\leq i\leq r$ the maximality of $P_i$ implies 
\begin{equation}
P_i=N_G\left(\fraku_{\alpha_i}\right)\label{eq:p_i}. 
\end{equation}

For any $1\leq i\leq r$ let $V_i=\bigwedge^{d_i}\frakg$ and \[
v_i=X_{1}\wedge\dots\wedge X_{d_i}\in V_i
\]
where
\begin{equation}\label{B'i}
	\left\{X_1,\dots,X_{d_i}\right\}=\bigcup_{\beta\in\Phi_{\alpha_i}}\calb_{\beta}=:\calb_i^\prime
\end{equation}  is a basis for $\fraku_{\alpha_{i}}$ ($v_i$ is uniquely determined up to a sign).

Fix a norm on $\frakg$. 

The following corollary can be deduced in a similar way to \cite[Cor. 3.3]{key-18} using Proposition \ref{Compactness criterion}, Proposition \ref{prop: unique}, and \eqref{eq:p_i}. 

\begin{cor}\label{cor:active element}
	For any $\epsilon>0$ there exists a neighborhood $W_\epsilon$ of zero in $\frakg$ such that if $W\subset W_\epsilon$ is a neighborhood of zero and $\gamma\in\Gamma$ is $\left(W,i\right)$-active for  $g\in G$, then 
	\[
	\norm{\varrho_{i}\left(g\gamma\right)v_i}<\epsilon.
	\]
\end{cor}

\begin{prop}
\label{prop: unbounded element} There exists a neighborhood $W_0$ of zero in $\frakg$ such that any neighborhood of zero $W_1\subset W_0$ satisfies the following.
Let $g\in G$, $1\leq i\leq r$,
$\Psi\subset\Phi_{\varrho_i}$, and $A$ be a connected subset of $T$
such that $ag$ is $\left(W_1,i\right)$-active for any $a\in A$.
If for any neighborhood $W$ of zero in $\frakg$ and $\lambda\in\Psi$ there is $a\in A$ such that
$ag$ is $\left(W,i\right)$-active and $\lambda\left(a\right)\geq1$,
then there exists $\gamma\in\Gamma$ such that 
\[
\varrho_{i}\left(g\gamma\right)v_{i}\in\bigoplus_{\lambda\in\Phi_{{\varrho_i}}\setminus\Psi}V_{\varrho_{i},\lambda}.
\]
\end{prop}
\begin{proof}
Let $W_0$ be as in Proposition \ref{prop: unique}. 
For $\lambda\in\Phi_{\varrho_i}$ let $\varphi_{\lambda}:V_i\rightarrow V_{\varrho_i,\lambda}$
be the projection associated with the direct sum decomposition $V_i=\bigoplus_{\lambda\in\Phi_{\varrho_i}}V_{\varrho_i,\lambda}$. 

Let $\lambda\in\Psi$. By Proposition \ref{Compactness criterion}, Proposition \ref{prop: unique}, Corollary \ref{cor:active element}, and the assumption, there exist $\gamma\in\Gamma$ and a sequence $\left\{a_{j}\right\}\subset A$ such that 
\begin{equation}
\lambda\left(a_{j}\right)\geq1\mbox{ and }\norm{\varrho_{i}\left(a_{j}g\gamma\right)v_{i}}\underset{j\rightarrow\infty}{\longrightarrow}0.\label{eq: norm going to 0}
\end{equation}
Since any norm on a finite-dimensional vector space is equivalent
to the sup-norm, there exists $c>0$ so that 
\begin{equation}
\norm{\varrho_{i}\left(a_{j}g\gamma\right)v_{i}}\geq c\cdot\norm{\varphi_{\lambda}\left(\varrho_{i}\left(a_{j}g\gamma\right)v_{i}\right)}=c\cdot\lambda\left(a_{j}\right)\norm{\varphi_{\lambda}\left(\varrho_{i}\left(g\gamma\right)v_{i}\right)}.\label{eq: sup-norm bound}
\end{equation}

Equations \eqref{eq: norm going to 0} and \eqref{eq: sup-norm bound} prove the conclusion of the proposition. 
\end{proof}
\begin{thm}
\label{thm: conn component intersects both rays}Assume that $\mbox{rank}_{\bbr}G=\mbox{rank}_{\bbq}G=2$ and $\Phi_\bbq$ is irreducible.
For any $g\in G$ there exists a neighborhood $W_0=W_0\left(g\right)$ of zero in $\frakg$ such that any neighborhood of zero $W\subset W_0$ satisfies the following. Let $\fraka$ be a connected subset of 
\begin{equation}
\fraka^{-}=\left\{ t\in\mathfrak{t}\::\:\chi_{1}\left(t\right)\leq0,\:\omega_{\alpha_{1}}\left(\chi_{1}\right)\left(t\right)\leq0\right\} \label{eq: a- def}
\end{equation}
 such that $\exp\left(a\right)g$ is $\left(W,1\right)$-active
for any $a\in\fraka$. Then there exists $\chi\in\left\{ \chi_{1},\omega_{\alpha_{1}}\chi_{1}\right\} $
such that all $a\in\fraka$ satisfy $\chi\left(a\right)<0$. \end{thm}
\begin{proof}
Let the notation be as in the proof of Proposition \ref{prop: unbounded element}. 

Assume $\calb=\left\{X_1,\dots,X_d\right\}$ and denote the $\bbz$-span of \[\left\{X_{i_1}\wedge\cdots\wedge X_{i_{d_1}}\::\:1\leq i_i<\cdots<i_{d_1}\leq d\right\}\subset V_1\]
by $V_1\left(\bbz\right)$.
Then, $\varrho_{1}\left(g\right)V_{1}\left(\bbz\right)$ is a discrete
subset. Therefore, there exists $\epsilon_{0}=\epsilon_{0}\left(g\right)>0$
such that 
\begin{equation}
\forall v\in\varrho_{1}\left(g\right)V_{1}\left(\bbz\right)\qquad\norm v>\epsilon_{0}.\label{eq: bigger epsilon}
\end{equation}

Let $c$ satisfy the conclusions of Theorem \ref{thm:weyl elem norm bound}. Let $W$ satisfy the conclusion of Corollary \ref{cor:active element} for $\epsilon=\frac{\epsilon_0}{2c}$. Without loss of generality assume $W\subset W_0$. 
Then, according to Corollary \ref{cor:active element} and Proposition \ref{prop: unique} there is $\gamma\in\Gamma$ such that for all $a\in\fraka$
\begin{equation}
\norm{\varrho_{1}\left(\exp\left(a\right)g\gamma\right)v_{1}}<\epsilon=\frac{\epsilon_{0}}{2c}.\label{eq: smaller epsilon}
\end{equation}

For any $\lambda\in\Phi_{\varrho_1}$ we have 
\begin{equation}
\varphi_{\lambda}\left(\varrho_{1}\left(\exp\left(a\right)g\gamma\right)v_{1}\right)=\exp\left(\lambda\left(a\right)\right)\varphi_{\lambda}\left(\varrho_{1}\left(g\gamma\right)v_{1}\right).\label{eq: after exp}
\end{equation}

Let $s\in W\left(\Phi_{\bbr}\right)$ satisfy 
\[
\norm{\varphi_{s\left(\chi_{1}\right)}\left(\varrho_{1}\left(g\gamma\right)v_{1}\right)}=\max_{\omega\in W\left(\Phi_{\bbr}\right)}\norm{\varphi_{\omega\left(\chi_{1}\right)}\left(\varrho_{1}\left(g\gamma\right)v_{1}\right)}.
\]
Then, by Theorem \ref{thm:weyl elem norm bound} and (\ref{eq: bigger epsilon})
\begin{equation}
\norm{\varphi_{s\left(\chi_{1}\right)}\left(\varrho_{1}\left(g\gamma\right)v_{1}\right)}>\frac{\epsilon_{0}}{c}.\label{eq: proj chi big}
\end{equation}
By Lemma \ref{lem: not both pos}, there exist $b_{1},b_{2}$, not
both positive, such that $s\left(\chi_{1}\right)=b_{1}\chi_{1}+b_{2}\omega_{\alpha_{1}}\left(\chi_{1}\right)$.
If $b_{1}\leq0$ and there exists $a\in\fraka$ such that $\omega_{\alpha_{1}}\chi_{1}\left(a\right)=0$,
then by (\ref{eq: a- def}), $s\left(\chi_{1}\right)\left(a\right)\geq0$.
Using (\ref{eq: after exp}), we arrive at 
\[
\norm{\varphi_{s\left(\chi_{1}\right)}\left(\varrho_{1}\left(\exp\left(a\right)g\gamma\right)v_{1}\right)}\geq\norm{\varphi_{s\left(\chi_{1}\right)}\left(\varrho_{1}\left(g\gamma\right)v_{1}\right)}.
\]
But then, using (\ref{eq: smaller epsilon}) and (\ref{eq: proj chi big})
we obtain a contradiction. Hence, all $a\in\fraka$ satisfy $\omega_{\alpha_{1}}\chi_{1}\left(a\right)<0$.
In a similar way, if $b_{2}\leq0$, then all $a\in\fraka$ satisfy
$\chi_{1}\left(a\right)<0$. 
\end{proof}

\section{\label{sec:Proof-of-Theorem2}Proof of Theorem \ref{thm: all obvoius}}

We keep the notation of \S \ref{sec:Compactness-criterion}.

For the proof of Theorem \ref{thm: all obvoius} we will need the
following results:
\begin{thm}[Lebesgue]
\label{thm:Lebesgue}\cite[Thm 4.3]{key-4} Let the two-dimensional
unit cube $\mathcal{I}=\left[0,1\right]^{2}$ be covered by a pair
of closed sets $X_{1},X_{2}$. For $i=1,2$ let $F_{i}^{+}$ and $F_{i}^{-}$
be the facets of $\mathcal{I}$ defined by $x_{i}=1$ and $x_{i}=0$
respectively. Then some connected component of $X_{i}$, $i\in\left\{ 1,2\right\} $,
intersects both the corresponding opposite facets $F_{i}^{+}$ and
$F_{i}^{-}$. 
\end{thm}

\begin{thm}[Riemann]
\label{thm:Riemann}\cite[\S II]{key-20} Let $D$ denote a simply
connected domain in $\bbc$ that has more than one boundary point.
Then, there exists an analytic one-to-one mapping of $D$ onto the
disk $\left\{ \left|z\right|<1\right\} $. If the boundary of $D$
is a piecewise smooth curve, then the mapping extends in a unique
analytic one-to-one way to the closure of $D$.
\end{thm}
Replacing $\Delta_{\bbq}$ with $-\Delta_{\bbq}$, we may assume 
\[
\mathfrak{a}^{+}=\left\{ t\in\mathfrak{t}\::\:\chi_{1}\left(t\right)\leq0,\;\omega_{\alpha_{1}}\left(\chi_{1}\right)\left(t\right)\leq0\right\} .
\]

Recall $A^{+}=\exp\left(\mathfrak{a}^{+}\right)$. Let $A^{+}\subset A\subset T$
be a closed cone and let $g\in G$. Suppose that $A\pi\left(g\right)$
is a divergent trajectory. We need to prove that it is an obvious divergent trajectory. 

Let $W_0$ satisfy the conclusions of Proposition \ref{prop: unbounded element} and Theorem \ref{thm: conn component intersects both rays}. Let $W_1$ be an open neighborhood of zero in $\frakg$ so that its closer satisfies  $\overline{W_1}\subset W_0$. 

Fix a norm on $\frakg$ and denote 
\begin{eqnarray*}
	&  & \fraka^{+}\left(r,R\right)=\left\{ a\in\fraka^{+}\::\: r\leq\norm a\leq R\right\} ,\\
	&  & \fraka^{+}\left(r\right)=\left\{ a\in\fraka^{+}\::\:\norm a\geq r\right\} ,\\
	&  & \fraka_{0}^{+}\left(r\right)=\left\{ a\in\fraka^{+}\::\:\norm a=r\right\} .
\end{eqnarray*}

According to Proposition \ref{Compactness criterion} and Remark \ref{rem:compact} there exists a compact subset $C\subset A$ such that for any
$a\in A\setminus C$ there are $1\leq i\leq r$ and $\gamma\in\Gamma$ which is $\left(W_1,i\right)$-active for $ag$.
For $i=1,2$ denote\[
\frakd_{i}=\left\{a\in \fraka\::\: \exists\gamma\in\Gamma\mbox{ which is }\left(W_1,i\right)\mbox{-active for }\exp\left(a\right)g\right\}.
\label{eq: d_i def}
\]
Then, there exists $r_0>0$ such that $\frakd_{1}$, $\frakd_{2}$
is a cover of $\fraka^{+}\left(r_{0}\right)$. Moreover, it follows from Definition \ref{def:active} that for $i=1,2$ the set $\frakd_{i}$ is open in $\fraka=\log\left(A\right)$ and satisfies
\begin{equation}
\overline{\frakd}_i\subset\left\{a\in \fraka\::\: \exists\gamma\in\Gamma\mbox{ which is }\left(W_0,i\right)\mbox{-active for }\exp\left(a\right)g\right\}.
\label{eq:d_i subset}
\end{equation} 

We first claim 
\begin{equation}
\forall r>r_{0}\;\exists\mathcal{P}\mbox{ an unbounded connected component of }\overline{\frakd}_{2}\mbox{ such that }\mathcal{P}\cap\fraka_{0}^{+}\left(r\right)\neq\emptyset.\label{eq: unbounded conn comp}
\end{equation}
Assume otherwise. Let $r_1>r_{0}$.
Since $\frakd_{2}$ is an open set and $\fraka^{+}\left(r_1,r_1+1\right)$
is a compact set there are only finitely many connected components of $\frakd_{2}$ which intersect $\fraka^{+}\left(r_1,r_1+1\right)$
and the set 
\[
S_{1}=\left\{ \mathcal{P}\mbox{ is a connected component of }\overline{\frakd_{2}},\;\mathcal{P}\cap\mathfrak{a}_{0}^{+}\left(r_1\right)\neq\emptyset\right\} 
\]
 is finite. Let $R_{1}>\max\left(\norm a\::\:\exists\mathcal{P}\in S_{1}\mbox{ s.t. }a\in\mathcal{P}\right)$.
Then there is no connected component of $\overline{\frakd_{2}}$ which
intersects both $\mathfrak{a}_{0}\left(r_{1}\right),\mathfrak{a}_{0}\left(R_{1}\right)$.

The boundary of $\fraka^{+}$ is the union of the following two rays
from the origin
\[
\left\{ a\in\fraka^{+}\::\:\chi_{1}\left(a\right)=0\right\} \;\mbox{ and }\;\left\{ a\in\fraka^{+}\::\:\omega_{\alpha_{1}}\chi_{1}\left(a\right)=0\right\} .
\]
Denote
\begin{equation}
\begin{array}{l}
\mathfrak{l}_{1}\left(r,R\right)=\left\{ a\in\mathfrak{a}^{+}\left(r,R\right)\::\:\chi_{1}\left(a\right)=0\right\} ,\\
\mathfrak{l}_{2}\left(r,R\right)=\left\{ a\in\mathfrak{a}^{+}\left(r,R\right)\::\:\omega_{\alpha_{1}}\chi_{1}\left(a\right)=0\right\} .
\end{array}\label{eq: l_i def}
\end{equation}

Fix a homeomorphism from $\mathfrak{a}^{+}\left(r_{1},R_{1}\right)$
to the two-dimensional unit cube $\mathcal{I}$ which sends $\mathfrak{l}_{1}\left(r_{1},R_{1}\right),\mathfrak{l}_{2}\left(r_{1},R_{1}\right)$
to $F_{1}^{+},F_{1}^{-}$ and $\mathfrak{a}_{0}\left(r_{1}\right),\mathfrak{a}_{0}\left(R_{1}\right)$
to $F_{2}^{+},F_{2}^{-}$. Then we may apply Theorem \ref{thm:Lebesgue}
with $\overline{\mathcal{\frakd}_{1}}$, $\overline{\frakd_{2}}$
as the cover of $\mathfrak{a}^{+}\left(r_{1},R_{1}\right)$
to deduce that there is a connected component of $\overline{\frakd_{1}}$
which intersects both $\mathfrak{l}_{1}\left(r_{1},R_{1}\right)$,
$\mathfrak{l}_{2}\left(r_{1},R_{1}\right)$, a contradiction
to Theorem \ref{thm: conn component intersects both rays}, proving  \eqref{eq: unbounded conn comp}. 

Let $\mathcal{D}_{1},\mathcal{D}_{2}$ be the sets of all unbounded connected components of $\overline{\frakd_{1}},\overline{\frakd_{2}}$,
respectively. Let $\mathcal{D}=\mathcal{D}_{1}\cup\mathcal{D}_{2}$.
Then by (\ref{eq: unbounded conn comp}) $\mathcal{D}$ is not empty. 

Recall that for $i=1,2$, $\varrho_i:G\rightarrow\mbox{GL}\left(V_i\right)$ is a strongly rational over $\bbq$-representation with a set of $\bbr$-weights $\Phi_{{\varrho_i}}$, an $\bbr$-highest weight $\chi_i$, and $v_i$ a non-zero $\bbr$-weight vectors for $\chi_i$. For $\lambda\in\Phi_{\varrho_i}$ let $\varphi_\lambda:V_i\rightarrow V_{\varrho_{i},\lambda}$ be the
projection associated with the direct sum decomposition $V_i=\bigoplus_{\lambda\in\Phi_{\varrho_i}}V_{\varrho_i,\lambda}$. 

Let $i=1,2$ and $\mathcal{P}\in\mathcal{D}_i$. It follows from \eqref{eq:d_i subset} Proposition \ref{Compactness criterion}, Proposition \ref{prop: unbounded element}, and our assumption that $A\pi\left(g\right)$ diverge that there exists $\gamma_\mathcal{P}\in\Gamma$ such that for any unbounded $\left\{a_{k}\right\}\subset A$
\[
\varrho_{i}\left(ag\gamma_\mathcal{P}\right)v_{i}\rightarrow 0\mbox{ as }k\rightarrow\infty
\]
Denote $v_{\mathcal{P}}=\varrho_{i}\left(g\gamma_\mathcal{P}\right)v_{i}$. Denote by $\Psi_{\mathcal{P}}$ the set of weights $\lambda\in\Phi_{\varrho_{i}}$ such that $\varphi_\lambda\left(\varrho_{i}\left(g\gamma_\mathcal{P}\right)v_{i}\right)
\neq0$. 

Let 
\begin{eqnarray*}
 &  & \Upsilon_{i}=\left\{ \Psi_{\mathcal{P}}\;:\;\mathcal{P}\in\mathcal{D}_{i}\right\} ,\quad i=1,2,\\
 &  & \mbox{and }\quad\Upsilon=\Upsilon_{1}\cup\Upsilon_{2}
\end{eqnarray*}
Since both $\Phi_{\varrho_1}$, $\Phi_{\varrho_2}$ are finite, $\Upsilon_{1},\Upsilon_{2}$ are also finite. 

If for any $a\in\fraka$ there exists $\Psi\in\Upsilon$ such that
$\lambda\left(a\right)<0$ for all $\lambda\in\Psi$, then
we are done. Indeed, let $\left\{ a_{k}\right\}\subset\fraka$ be an unbounded sequence. 
Since $\Upsilon$ is finite, there are $i=1,2$ and $\Psi\in\Upsilon_i$ such that for some subsequence $\left\{ a_{k_{\ell}}\right\} $, $\Psi_{a_{k_{\ell}}}=\Psi$. Then, for $\mathcal{P}\in\mathcal{D}_{i}$ which satisfies $\Psi=\Psi_{\mathcal{P}}$ we have
\begin{equation}
\norm{\varrho_{i}\left(\exp\left(a_{k_{\ell}}\right)\right)v_{\mathcal{P}}}\leq\max_{\lambda\in\Psi}\left(\exp\left(\lambda\left(a_{k_{\ell}}\right)\right)\right)\cdot\norm{v_{\mathcal{P}}}\rightarrow0\quad\mbox{as}\quad k\rightarrow\infty.\label{eq: end of proof}
\end{equation} 
Note that it follows from \eqref{eq:p_i} that the second part of Definition \ref{degenerate} is satisfied for any $v_\mathcal{P}$, $\mathcal{P}\in\mathcal{D}$. 

Otherwise; there exists $a_{0}\in\fraka$ such
that $\norm{a_{0}}=1$ and for any $\Psi\in\Upsilon$ there exists $\lambda\in\Psi$
such that $\lambda\left(a\right)\geq0$. 

Denote 
\begin{eqnarray*}
 &  & \fraka\left(r,R\right)=\left\{ a\in\fraka\::\: r\leq\norm a\leq R\right\},\\
 &  & \fraka\left(r\right)=\left\{ a\in\fraka\::\:\norm a\geq r\right\},\\
 &  & \fraka_{0}\left(r\right)=\left\{ a\in\fraka\::\:\norm a=r\right\},
\end{eqnarray*}
and
\[
\epsilon_0=\min\left\{ \norm{\varphi_{\lambda}\left(v_{\Psi}\right)}\::\:\Psi\in\Upsilon\right\}.
\]
Since $\Phi_{\varrho_1}$, $\Phi_{\varrho_2}$ are finite sets, and $g\Gamma$ is a discrete set, $\epsilon_0$ is non-zero. 

Let $W_2\subset W_0$ be an open set which satisfies the conclusion of Corollary \ref{cor:active element} for $0<\epsilon<\epsilon_0$.
We can repeat the above arguments with $W_2$ instead of $W_0$ (and maybe get a larger $\epsilon_0$).  
In that case, by Proposition \ref{prop: unbounded element} any $\mathcal{P}\subset\mathcal{D}$ satisfies\[
\mathcal{P}\cap\fraka\left(r_{1}\right)\subset\left\{ a\in\fraka\left(r_{1}\right)\::\:\forall\lambda\in\Psi_{\mathcal{P}},\quad\lambda\left(a\right)<0\right\}.
\]
Hence, our choise of $\epsilon$ implies
\begin{equation}
\forall\mathcal{P}\in\mathcal{D},\quad\mathcal{P}\cap\left\{ ta_{0}\::\: t\geq r_{1}\right\} =\emptyset.
\end{equation}

Since $\fraka$ is a closed cone, there exist $\mathfrak{l}_{1},\mathfrak{l}_{2}\subset\fraka$
two disjoint rays from the origin so that their union is the boundary
of $\fraka$. Let $\mathcal{P}_{1},\mathcal{P}_{2}\in\mathcal{D}\cup\left\{ \mathfrak{l}_{1},\mathfrak{l}_{2}\right\} $
be the sets adjacent to the line $\left\{ ta_{0}\::\: t\geq r_{1}\right\} $
on the line $\fraka_{0}\left(r_{2}\right)$. Note that by (\ref{eq: unbounded conn comp})
we may assume $\mathcal{P}_{1}\in\mathcal{D}$. Let $\mathfrak{r}$
be the maximal closed connected subset of $\fraka\left(r_{1}\right)$
such that $\left\{ ta_{0}\::\: t\geq r_{1}\right\} \subset\mathfrak{r}$
and $\mathfrak{r}$ does not intersect the interior of $\mathcal{P}_{1}\cup\mathcal{P}_{2}$.
See Figure \ref{fig:2}. 

\begin{figure}[h]
\includegraphics[scale=0.7]{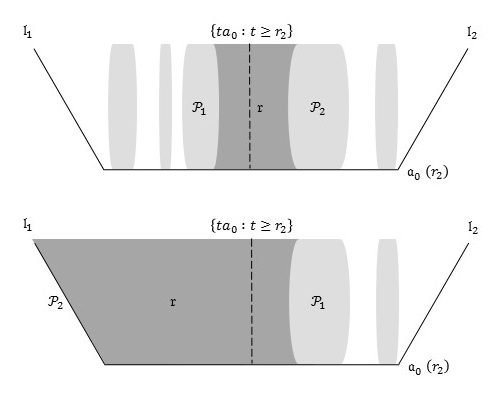}\protect\caption{\label{fig:2}Either $\mathfrak{r}$ has elements of $\mathcal{D}$
on both sides, or one of the boundaries of $\mathfrak{r}$ is a boundary
of the cone.}
\end{figure}

Without loss of generality assume $\mathcal{P}_{1}\subset\overline{\frakd_{2}}$.
As before, since $\frakd_{2}$ is an open set and $\fraka\left(r_{1},r_{1}+1\right)$
is a compact set there are only finitely many connected components
of $\frakd_{1}$ which intersect $\fraka\left(r_{1},r_{1}+1\right)$
and the set 
\[
S_{2}=\left\{ \mathcal{P}\mbox{ is a connected component of }\overline{\frakd_{1}},\;\mathcal{P}\cap\mathfrak{r}\cap\mathfrak{a}_{0}\left(r_{2}\right)\neq\emptyset\right\} \setminus\left\{ \mathcal{P}_{2}\right\} 
\]
 is finite. Let $R_{2}>\max\left(\norm a\::\:\exists\mathcal{P}\in S_{2}\mbox{ s.t. }a\in\mathcal{P}\right)$.
Then there is no connected component of $\overline{\frakd_{1}}$ which
intersects both $\mathfrak{a}_{0}\left(r_{1}\right),\mathfrak{a}_{0}\left(R_{2}\right)$.

By the definition of $\mathfrak{r}$, $\mathfrak{r}\cap\mathfrak{a}\left(r_{1},R_{2}\right)$
is a bounded simply connected domain. Let 
\begin{eqnarray*}
 &  & F_{1}^{-}=\mathfrak{r}\cap\fraka_{0}\left(r_{1}\right),\quad F_{1}^{+}=\mathfrak{r}\cap\mathfrak{a}_{0}\left(R_{2}\right).\\
 &  & F_{2}^{-}=\mathfrak{r}\cap\overline{\mathcal{P}_{1}}\cap\mathfrak{a}\left(r_{1},R_{2}\right),\quad F_{2}^{+}=\mathfrak{r}\cap\overline{\mathcal{P}_{2}}\cap\mathfrak{a}\left(r_{1},R_{2}\right),
\end{eqnarray*}
Using Theorem \ref{thm:Riemann}, one can construct a sequence of
sets $I_{k}\subset\mathfrak{r}\cap\mathfrak{a}\left(r_{1},R_{2}\right)$
which satisfy the following:
\begin{enumerate}
\item The boundary of each $I_{k}$ is a disjoint union (up to the end points)
of four simple smooth curves $F_{1}^{\pm}\left(I_{k}\right),F_{2}^{\pm}\left(I_{k}\right)$,
such that for $i=1,2$, $F_{i}^{+}\left(I_{k}\right)$ and $F_{i}^{-}\left(I_{k}\right)$
are opposite. 
\item Any $\sigma\in\left\{ +,-\right\} $, $i\in\left\{ 1,2\right\} $,
satisfy $F_{i}^{\sigma}\left(I_{k}\right)\rightarrow F_{i}^{\sigma}$
as $k\rightarrow\infty$. 
\end{enumerate}
Moreover, Theorem \ref{thm:Riemann} implies each $I_{k}$ is conformally
equivalent to $\mathcal{I}$ with facets $F_{i}^{\sigma}$, $\sigma\in\left\{ +,-\right\} $,
$i\in\left\{1,2\right\} $. Since $\mathfrak{r}$ is covered by $\frakd_{1}\cup\frakd_{2}$,
so is $I_{k}$, $k\in\bbn$. Therefore, for any $k\in\bbn$ we may
use Theorem \ref{thm:Lebesgue} with $\overline{\frakd_{1}},\overline{\frakd_{2}}$
as a closed cover of $I_{k}$ with facets $F_{1}^{\pm},F_{2}^{\pm}$,
to deduce that for some $i\in\left\{1,2\right\} $ there exists a
connected component $\mathcal{P}_{i,k}$ of $\overline{\frakd_{i}}$
which intersects both $F_{i}^{\pm}\left(I_{k}\right)$. Then there
exists $i\in\left\{1,2\right\} $ which appears infinitely many times
in the sequence $i\left(1\right),i\left(2\right),\dots$. Since $\mathfrak{a}\left(r_{1},R_{2}\right)$
is compact, there exists a nonempty limit set $\mathcal{P}_{i,k}\rightarrow\mathcal{P}_{3}$.
Then, by the definition of $R_{2}$, $\mathcal{P}_{3}\subset\overline{\frakd_{2}}$
and intersects $F_{2}^{\pm}$. Since any connected component of $\overline{\frakd_{2}}$
which intersects $\mathcal{P}_{1}$ is $\mathcal{P}_{1}$, we may
deduce $\mathcal{P}_{1}$ intersects $\mathcal{P}_{2}$ in $\mathfrak{r}$,
a contradiction to the definition of $\mathfrak{r}$.

\section{\label{sec:Proof-of-Theorem 1}Proof of Theorem \ref{thm: exists non-obvious}}

Replacing $\Delta_{\bbq}$ with $-\Delta_{\bbq}$, we may assume 
\begin{equation}
\mathfrak{a}_{\epsilon}=\left\{ t\in\mathfrak{t}\::\:\chi_{1}^{\prime}\left(t\right)\leq-\epsilon\norm t,\;\chi_{2}^{\prime}\left(t\right)\leq-\epsilon\norm t\right\}\label{eq:a- def}
\end{equation}
for some $\epsilon>0$.

We preserve the notation of \S \ref{sec:Compactness-criterion}.

\begin{thm}
	\label{thm: weiss for epsilon case}\cite[Thm. 4.5]{key-5} Suppose
	$G$ is a semisimple $\bbq$-algebraic group, $\Gamma=G\left(\bbz\right)$,
	and $A\subset G$ is a closed cone. Suppose that for $i=1,2$ there
	are $\bbq$-representations $\varsigma_{i}:G\rightarrow\mbox{GL}\left(W_{i}\right)$
	and non-zero vectors $w_{i}\in W_{i}\left(\bbq\right)$ such that the
	following hold: 
	\begin{enumerate}
		\item For any divergent (in $G$) sequence $\left\{a_{k}\right\}\subset A$
		we have $\varsigma_{i}\left(a_{k}\right)w_{i}\underset{k\rightarrow\infty}{\longrightarrow}0$
		for both $i$. 
		\item The groups $Q_{i}=\left\{ g\in G\::\:\bar{\varrho}_{i}(g)w_{i}\in\bbr w_{i}\right\}$,
		$i=1,2$, are $\bbq$-parabolic subgroups of $G$ and $Q_{1},\:Q_{2}$
		generate $G$. 
	\end{enumerate}
	Then there is $x\in G$ such that $A\pi\left(x\right)$ is divergent,
	but for any one-parameter semigroup $\left\{ \exp\left(ta\right)\::\: t\geq0\right\} \subset A$,
	any $\bbq$-representation $\varsigma:G\rightarrow\mbox{GL}\left(W\right)$
	and any non-zero $w\in W\left(\bbq\right)$ we have 
	\[
	\varsigma\left(\exp\left(ta\right)x\right)w\underset{t\rightarrow+\infty}{\not\longrightarrow}0.
	\]
	In particular there are non-obvious divergent trajectories for $A$.
\end{thm}

In the notations of \S \ref{sec:Compactness-criterion} let $w_{i}=v_i$ and $\varsigma_{i}=\varrho_{i}$ for any $i=1,2$. In order to prove
Theorem \ref{thm: exists non-obvious} it is enough to prove that $v_{1},v_{2},\varrho_{1}$,
and $\varrho_{2}$ satisfy conditions $\left(1\right)$ and
$\left(2\right)$ of Theorem \ref{thm: weiss for epsilon case} for
any $A\subset A_{\epsilon}$.

Let $\left\{\exp\left(a_{k}\right)\right\}\subset A$
be a divergent sequence. Then $\norm{a_{k}}\rightarrow\infty$ as
$k\rightarrow\infty$. Without loss of generality we may assume that
the norm defined on $\frakt$ is the sup-norm defined using a basis
which contains only $\bbq$-weight vectors. Thus, using \eqref{eq: V_chi_beta def} and \eqref{eq: a- def} for $i=1,2$ we have
\[
\norm{\varrho_{i}\left(\exp\left(a_{k}\right)\right)v_{i}}=e^{\tilde{\chi}_{i}\left(a_{k}\right)}\norm{v_{i}}<e^{-\epsilon\norm{a_{k}}}\norm{v_{i}}\underset{k\rightarrow\infty}{\longrightarrow}0,
\]
hence $\left(1\right)$ is satisfied. 

Since $P_1,P_2$ are maximal $\bbq$-parabolic subgroups of $G$, according to  \eqref{eq:p_i} in the notation of Theorem \ref{thm: weiss for epsilon case} we have $Q_{i}=P_{i,\bbq}$ for $i=1,2$. The maximality of $P_1,\:P_2$ implies that condition $\left(2\right)$ is also satisfied.

\end{document}